\documentclass{article}

\usepackage{amsmath, amssymb, amsthm, enumitem,appendix} 

\allowdisplaybreaks
\expandafter\let\expandafter\oldproof\csname\string\proof\endcsname
\let\oldendproof\endproof
\renewenvironment{proof}[1][\proofname]{%
	\oldproof[\bf #1]%
}{\oldendproof}

\allowdisplaybreaks

\theoremstyle{plain}
\newtheorem{theorem}{Theorem}
\newtheorem{lemma}{Lemma}[section]

\newtheorem{proposition}[lemma]{Proposition}

\newtheorem{corollary}[theorem]{Corollary}

\newtheorem{problem}[lemma]{Problem}

\newtheorem{definition}[lemma]{Definition}

\newcommand{\whp}{w.h.p.}

\newcommand{\D}{\mathcal D}

\newcommand{\U}{\mathcal U}

\newcommand{\poly}{\text{poly}}
\RequirePackage[normalem]{ulem} 
\RequirePackage{color}\definecolor{RED}{rgb}{1,0,0}\definecolor{BLUE}{rgb}{0,0,1} 

\title{Testing Graphs against an Unknown Distribution\footnote{A preliminary version of this paper has appeared in the Proceedings of STOC '19.}}

\author{Lior Gishboliner \thanks{School of Mathematics, Tel Aviv University, Tel Aviv 69978, Israel. Email: liorgis1@post.tau.ac.il. Supported in part by ERC Starting Grant 633509.}
\and Asaf Shapira \thanks{
School of Mathematics, Tel Aviv University, Tel Aviv 69978, Israel.
Email: asafico$@$tau.ac.il. Supported in part by ISF Grant 1028/16 and ERC Starting Grant 633509.}
}

\date{}

\begin{document}

\maketitle

\begin{abstract}
The area of graph property testing seeks to understand the relation between the global properties of a graph and its local statistics.
In the classical model, the local statistics of a graph is defined relative to a uniform distribution over the graph’s vertex set.
A graph property ${\cal P}$ is said to be testable if the local statistics of a graph can allow one to distinguish between graphs satisfying ${\cal P}$ and those that are far from satisfying it.

Goldreich recently introduced a generalization of this model in which one endows the vertex set of the input graph with an arbitrary
and unknown distribution, and asked which of the properties that can be tested in the classical model can also be tested in this more general setting.
We completely resolve this problem by giving a (surprisingly ``clean'') characterization of these properties. To this end, we prove a removal lemma
for vertex weighted graphs which is of independent interest.

\end{abstract}


\section{Introduction}\label{sec:intro}

\subsection{Background and the main result}

Property testers are fast randomized algorithms whose goal is to distinguish (with high probability, say, $2/3$) between objects
satisfying some fixed property ${\cal P}$ and those that are $\varepsilon$-far from satisfying it. Here, $\varepsilon$-far means that
an $\varepsilon$-fraction of the input object should be modified in order to obtain an object satisfying ${\cal P}$.
The study of such problems originated in the
seminal papers of Rubinfeld and Sudan \cite{Sudan}, Blum, Luby and Rubinfeld \cite{BLR}, and Goldreich, Goldwasser and Ron \cite{GGR}.
Problems of this nature have been studied in so many areas that it will be impossible to survey them here. Instead, the reader
is referred to the recent monograph \cite{Goldreich1} for more background and references. While this area studies questions
in theoretical computer science, it has several strong connections with central problems in extremal combinatorics, most notably to the
regularity method and the removal lemma, see Subsection \ref{subsec:LS}. 

The classical property testing model assumes that one can uniformly sample entries of the input.
In {\em distribution-free} testing, one assumes that the input is endowed with some {\em arbitrary and unknown} distribution ${\cal D}$,
which also affects the way one defines the distance to satisfying a property.
As discussed in \cite{Goldreich_VDF}, one motivation for this model is that it can handle settings in which one cannot produce uniformly distributed entries from the input. Another motivation is that the distribution ${\cal D}$ can assign higher weight/importance to parts of the input which we want to have higher impact on the distance to satisfying the given property.
Until very recently, problems of this type were studied almost exclusively in the setting of testing properties of functions,
see \cite{CLSSX,CX,DR,Glasner_Servedio,HK}. Let us mention that distribution-free testing is similar in spirit to the celebrated PAC learning model of Valiant \cite{Valiant}, see also the discussion in \cite{Ron}.

Our investigation here concerns a distribution-free variant of the
{\em adjacency matrix model}, also known as the {\em dense graph model}.
The adjacency matrix model was first defined and studied in
\cite{GGR}, where the area of property testing was first introduced.
This model has been extensively studied in the past two decades, see
Chapter $8$ of \cite{Goldreich1}. For a selected (but certainly not
comprehensive) list of works on the dense graph model of property
testing, see \cite{AFNS,GR_POT,GT}.

Instead of defining the adjacency matrix model of \cite{GGR}, let us
directly define its distribution-free variant which was introduced
recently
by Goldreich \cite{Goldreich_VDF}. Since the distribution in this
model is over the input's vertices, it is called the
{\em Vertex-Distribution-Free} (VDF) model\footnote{Goldreich
suggested to study variants of this model in other settings (such as
bounded degree graphs \cite{GR_bounded}) as well. For brevity,
we will use the term ``VDF model'' to refer to the ``VDF variant of
the adjacency matrix model''.}.
The input to the algorithm is a graph $G$ and some {\em arbitrary and unknown}
distribution ${\cal D}$ on $V(G)$.
We will thus usually refer to the input as the pair $(G,{\cal D})$.
For a pair of graphs $G_1,G_2$ on the same vertex-set $V$, and for a distribution $\mathcal{D}$ on $V$, the (edit) distance between $G_1$ and $G_2$ with respect to $\mathcal{D}$ is defined as
$\sum_{\{x,y\} \in E(G_1) \triangle E(G_2)}{\mathcal{D}(x)\mathcal{D}(y)}$. We say that $(G,{\cal D})$ is $\varepsilon$-far from satisfying a graph property\footnote{A {\em graph property is simply a family of graphs closed under isomorphism.}} ${\cal P}$ if for every $G' \in {\cal P}$, the distance between $G$ and $G'$ with respect to $\mathcal{D}$ is at least $\varepsilon$.
A tester for a graph property $\mathcal{P}$ is an algorithm that receives as input a pair $(G,\mathcal{D})$ and a proximity parameter $\varepsilon$, and distinguishes with high probability (say $\frac{2}{3}$) between the case that $G$ satisfies $\mathcal{P}$ and the case that $(G,\mathcal{D})$ is $\varepsilon$-far from $\mathcal{P}$. The algorithm has access to a device that produces random vertices from $G$ distributed according to ${\cal D}$.
The only\footnote{Note that the algorithm does not receive
$|V(G)|$ as part of the input.} other way the algorithm can access $G$ is by performing ``edge queries'' of the form ``is $(u,v)$ an edge of $G$?''.
We say that property ${\cal P}$ is {\em testable} in the VDF model if there is a function $q(\varepsilon)$ and a tester for $\mathcal{P}$
that always performs a total number of at most $q(\varepsilon)$ vertex samples and edge queries to the input. We stress again that ${\cal D}$ is unknown to the tester, so (in particular) that $q$ should be {\em independent} of ${\cal D}$.
The function $q$ is sometimes referred to as the sample (or query) complexity of the tester.
A tester has 1-sided error if it always accepts an input satisfying ${\cal P}$. Otherwise it has 2-sided error.

Suppose we assume that in the VDF model, the distribution ${\cal D}$ is restricted to be the uniform distribution; in particular, the distance between $n$-vertex graphs $G,G'$ (on the same vertex-set) is $|E(G) \triangle E(G')|/n^2$, and $G$ is $\varepsilon$-far from ${\cal P}$ if one needs to change at least $\varepsilon n^2$ edges to turn $G$ into a graph satisfying ${\cal P}$.
In this paper we will refer to this model as the {\em standard model}. This model is ``basically'' equivalent to the {\em adjacency matrix model}, which was introduced in \cite{GGR}. We refer the reader to \cite{Goldreich_VDF} for a discussion on the subtle differences between the {\em adjacency matrix model}
and the above defined standard model\footnote{Just as an example, in \cite{GGR} the tester ``knows'' $|V(G)|$ while in the VDF model (and thus also in the standard model) \nolinebreak it \nolinebreak does \nolinebreak not.}.

A very elegant result proved in \cite{Goldreich_VDF}, states that if ${\cal P}$ is testable in the VDF model then it is testable
in the standard model with one-sided error.
A natural follow-up question, raised by Goldreich in \cite{Goldreich_VDF}, asks whether every property which is testable with one-sided error in the standard model, is also testable in the VDF model.
A characterization of the properties testable with one-sided error in the standard model was given in \cite{AS_hereditary}, where it was shown that
these are precisely the {\em semi-hereditary} properties (see \cite{AS_hereditary} for the definition of this term).
We show (see Proposition \ref{prop:hard_hereditary}), that if $\mathcal{P}$ is testable in the VDF model then $\mathcal{P}$ is hereditary\footnote{A graph property is hereditary if it is closed under removal of vertices.}. Since there are properties which are semi-hereditary but not hereditary, this implies a negative answer to Goldreich's question.
Thus, it is natural to ask the following revised version of Goldreich's question:

\begin{problem}\label{prob:gold}
Are all hereditary graph properties testable in the VDF model?
\end{problem}

It might be natural to guess\footnote{This was at least our initial guess.} that every hereditary property is testable in the VDF model, the justification
being that all lemmas that were used in \cite{AS_hereditary} should also hold for weighted graphs. As it turns out, this is indeed the case. However, putting all these lemmas together does not seem to work in the VDF model.
As our main result, Theorem \ref{thm:main} below, shows, it is no coincidence that the proof technique of \cite{AS_hereditary} does not carry over as is to the weighted setting.

We start with an important definition. Let us say that a graph property $\mathcal{P}$ is {\em extendable} if for every graph $G$ satisfying $\mathcal{P}$ there is a graph $G'$ on $|V(G)|+1$ vertices which satisfies $\mathcal{P}$ and contains $G$ as an induced subgraph. In other words, $\mathcal{P}$ is extendable if whenever $G$ is a graph satisfying $\mathcal{P}$ and $v$ is a ``new" vertex (i.e. $v \notin V(G)$), one can connect $v$ to $V(G)$ in such a way that this larger graph will also satisfy $\mathcal{P}$. Note that if $\mathcal{P}$ is extendable then in fact for every graph $G \in \mathcal{P}$ and for every $n > |V(G)|$, there is an $n$-vertex graph satisfying $\mathcal{P}$ which contains $G$ as an induced subgraph. Our main result in this paper is the following:

\begin{theorem}\label{thm:main}
A graph property is testable in the VDF model if and only if it is hereditary and extendable.
\end{theorem}

It is interesting to compare the above (rather) simple characterization of the
properties that are testable in the VDF model, with the (very) complicated characterization of \cite{AFNS} of the properties that are testable
in the standard model.

Let us mention some immediate consequences of Theorem \ref{thm:main}.
Since a graph cannot contain both an isolated vertex and a vertex connected to all other vertices, we infer that
for every fixed $H$ the (hereditary) property of being induced $H$-free is extendable. We thus infer that:

\begin{corollary}\label{cor:indHfree}
The property of being induced $H$-free is testable in the VDF model for every fixed $H$.
\end{corollary}

It is also clear that the property of being $H$-free is extendable if and only if $H$ has no isolated vertices.
We thus infer that:

\begin{corollary}\label{cor:Hfree}
The property of being $H$-free is testable in the VDF model if and only if $H$ has no isolated vertices.
\end{corollary}

It is easy to see that most (natural) hereditary graph properties are extendable, so Theorem \ref{thm:main} immediately
implies that they are all testable in the VDF model. These include the properties of being Perfect, Interval, Chordal and $k$-Colorable.
In the other direction, Theorem \ref{thm:main} implies that if $H$ has an isolated vertex then $H$-freeness is not testable in the VDF model.
If one is interested in a more ``natural'' non-extendable hereditary property, then it is not hard to see that another such example is the
property ${\cal P}$ of being induced $\{A,B\}$-free, where $A$ (resp. $B$) is the graph obtained from the $2$-edge path $P_2$ by adding a new vertex which is adjacent to all $3$ vertices of $P_2$ (resp. not adjacent to any vertex of $P_2$). It is easy to see that $C_5$ satisfies ${\cal P}$
but is not extendable. It was proved in \cite{Goldreich_VDF} that the properties of being Hamiltonian, Eulerian and Connected are not testable in the VDF model. Those three results follow immediately from our Theorem \ref{thm:main} since these properties are not hereditary.

\subsection{The combinatorial interpretation of Theorem \ref{thm:main}}\label{subsec:LS}

Let us discuss the combinatorial implications of Theorem \ref{thm:main} and its relation to other results in the area of extremal combinatorics. 
The famous triangle removal lemma of Ruzsa and Szemer\'edi \cite{RuzsaSz76} states that if a graph $G$ is $\varepsilon$-far from being triangle free (with respect to the uniform distribution), then a (uniform) sample of $s(\varepsilon)$ vertices from $G$ contains a triangle with probability at least $\frac{2}{3}$.
We refer the reader to \cite{CF1} for more background on this lemma and its variants.
The result of \cite{AS_hereditary} mentioned above, can be thought of as a generalization of this lemma to arbitrary hereditary properties. It can be stated
as saying that for every hereditary graph property $\mathcal{P}$ there is a function $s_{\mathcal{P}} : (0,1) \rightarrow \mathbb{N}$ such that the following holds for every $\varepsilon > 0$. If a graph $G$ is $\varepsilon$-far from $\mathcal{P}$ (with respect to the uniform distribution) then a (uniform) sample
of $s_{\mathcal{P}}(\varepsilon)$ vertices from $G$ induces a graph not satisfying $\mathcal{P}$ with probability at least $2/3$.

To prove (the ``if" direction of) Theorem \ref{thm:main}, we will actually prove the following combinatorial statement, which can be thought of as a vertex-weighted version of the graph removal lemma.

\begin{theorem}\label{statement:ours}
	For every hereditary and extendable graph property $\mathcal{P}$ there is a function $s_{\mathcal{P}} : (0,1) \rightarrow \mathbb{N}$ such that the following holds for every $\varepsilon > 0$ and for every vertex-weighted graph $(G,\mathcal{D})$ which is $\varepsilon$-far from $\mathcal{P}$. Let $u_1,\dots,u_s$, $s = s_{\mathcal{P}}(\varepsilon)$, be a sequence of random vertices of $G$, sampled according to $\mathcal{D}$ and independently. Then $G[\{u_1,\dots,u_s\}]$ does not satisfy $\mathcal{P}$ with probability at least $\frac{2}{3}$.
\end{theorem}
\noindent
The following similar-looking result\footnote{We note that the results of \cite{AT} and \cite{LS} are more general. The authors of \cite{LS} actually prove that the conclusion of Theorem \ref{statement:LS} holds for all graphons. The authors of \cite{AT} prove extensions of Theorem \ref{statement:LS} in several directions, including a version for uniform hypergraphs, and a strengthening in which the notion of testability is replaced with the stronger notion of repairability.} was (implicitly) proved by Austin and Tao \cite{AT} and Lov\'{a}sz and \nolinebreak Szegedy \nolinebreak \cite{LS}.

\begin{theorem}[\cite{AT,LS}]\label{statement:LS}
	For every hereditary graph property $\mathcal{P}$ there is a function $s_{\mathcal{P}} : (0,1) \rightarrow \mathbb{N}$ such that the following holds for every $\varepsilon > 0$ and for every vertex-weighted graph $(G,\mathcal{D})$ which is $\varepsilon$-far from $\mathcal{P}$. Let $u_1,\dots,u_s$, $s = s_{\mathcal{P}}(\varepsilon)$, be a sequence of random vertices of $G$, sampled according to $\mathcal{D}$ and independently. Construct a graph $S$ on $[s]$ by letting $\{i,j\} \in E(S)$ if and only if $\{u_i,u_j\} \in E(G)$. Then $S$ does not satisfy $\mathcal{P}$ with probability at least $\frac{2}{3}$.
\end{theorem}

Note that Theorem \ref{statement:LS} holds for all hereditary properties, while Theorem \ref{statement:ours} only holds for hereditary properties which are extendable.
Observe that the graph $S$ in Theorem \ref{statement:LS} is a blowup of the graph $G[U]$, where $U = \{u_1,\dots,u_s\}$. Thus, the difference between Theorems \ref{statement:ours} and \ref{statement:LS} is that Theorem \ref{statement:LS} only guarantees that a {\em blowup} of $G[U]$ does not satisfy $\mathcal{P}$ w.h.p., while Theorem \ref{statement:ours} guarantees the stronger assertion that $G[U]$ itself does not satisfy $\mathcal{P}$ w.h.p. This is an important difference: while Theorem \ref{statement:ours} immediately implies the existence of a VDF-tester for every hereditary and extendable property $\mathcal{P}$ (see Subsection \ref{subsec:main_result}), we do not know of any way of using Theorem \ref{statement:LS} to prove the existence of such a tester.
One natural candidate for a tester derived from Theorem \ref{statement:LS} would be the algorithm which accepts if and only if the graph $S$ (defined in Theorem \ref{statement:LS}) does not satisfy $\mathcal{P}$. It turns out, however, that this algorithm often fails to be a valid tester\footnote{For example, if $\mathcal{P} = C_5$-freeness then this tester will reject w.h.p if the input graph is a triangle with uniform vertex distribution (as the graph $S$ will typically contain the 2-blowup of a triangle, and thus contain a copy of $C_5$), even though this input graph clearly satisfies \nolinebreak $\mathcal{P}$.}.

It is worth noting that Theorem \ref{statement:LS} can be deduced from the ``unweighted" case, i.e. the result of \cite{AS_hereditary}, via a simple argument, see Lemma \ref{lem:reduction} and the discussion following it. On the other hand, the proof of Theorem \ref{statement:ours} requires several new ideas on top of those used in \cite{AS_hereditary}.

\subsection{Variants of the VDF model}

The proof of the ``only if'' part of Theorem \ref{thm:main}, showing that if $\mathcal{P}$ is either non-extendable or non-hereditary
then $\mathcal{P}$ is not testable in the VDF model, relies on allowing the input graph to have only $O(1)$ vertices (where the constant is independent of $\varepsilon$); on excluding $|V(G)|$ from the input fed to the tester; and on having distributions $\mathcal{D}$ that assign to some vertices weight $\Theta(1)$ and
to some vertices weight $o(1/|V(G)|)$. This raises the natural question of what happens if we only require the tester to work on sufficiently large graphs; or if the tester receives $|V(G)|$ as part of the input; or if we forbid $\mathcal{D}$ from assigning
very low or very high weights (as above). As the following four theorems show, either one of these variations has a dramatic effect on the model, since it then
allows {\em all} hereditary properties to be testable.

We start with the setting in which the input graph is guaranteed to be large enough. In a revised version of \cite{Goldreich_VDF}, Goldreich asked whether every hereditary property $\mathcal{P}$ is testable (in the VDF model) on graphs of order at least $M = M_{\mathcal{P}}$, for $M$ which is independent of $\varepsilon$. As we show in Proposition \ref{prop:constant_bound_on_num_of_vertices}, this turns out to be false. On the positive side, we show that under the stronger assumption that the input size is at least $M_{\mathcal{P}}(\varepsilon)$ (where $M_{\mathcal{P}} : (0,1) \rightarrow \mathbb{N}$ is a function dependent on $\mathcal{P}$), {\em all} hereditary properties are testable.

\begin{theorem}\label{theo_rest3}
	Under the promise that $|V(G)| \gg 1$, every hereditary property is testable with one-sided error in the VDF model.
\end{theorem}


D. Ron (personal communication) asked what happens if we allow testers to receive $|V(G)|$ (i.e., the number of vertices in the input graph) as part of the input\footnote{We note that in the VDF model as defined in \cite{Goldreich_VDF}, the number of vertices in the input graph is not known to the tester}. Our following theorem answers \nolinebreak this \nolinebreak question.

\begin{theorem}\label{theo_rest4}
	If testers can receive $|V(G)|$ as part of the input, then every hereditary property is testable with one-sided error in the VDF model. 	
\end{theorem} 

\noindent
Finally, we consider settings in which restrictions are posed on the weights that the distribution $\mathcal{D}$ can \nolinebreak assign.

\begin{theorem}\label{theo:rest1}
	Under the promise that $\max_{v \in V(G)}{\mathcal{D}}(v) = o(1)$, every hereditary property is testable with one-sided error in the VDF model.
\end{theorem}

\begin{theorem}\label{theo:rest2}
	Under the promise that $\min_{v \in V(G)}{\mathcal{D}}(v) = \Omega\left(1/|V(G)|\right)$, every hereditary property is testable with one-sided error in the VDF model.
\end{theorem}

We note that the implied constant in the $\Omega$-notation in Theorem \ref{theo:rest2} is allowed to depend on $\varepsilon$.
We refer the reader to Section \ref{sec:restricted_models} for the precise statements of 
Theorems \ref{theo_rest3}--\ref{theo:rest2}.
Let us mention that the proofs of Theorems \ref{theo_rest3}, \ref{theo_rest4} and \ref{theo:rest2} rely on reductions to our main result in this paper,
Theorem 1. The proof of Theorem \ref{theo:rest1} proceeds by a reduction to the standard model (i.e. to the result of [5]). As part of this
proof, we solve another problem raised in [19]. 

\subsection{Paper overview}
The rest of the paper is organized as follows. Section \ref{sec:preliminary} is devoted to proving vertex-weighted analogues of several lemmas that were used in prior works (most notably regularity and counting lemmas, and corollaries thereof). Some more routine parts of these proofs are deferred to the appendix.
In Section \ref{sec:positive} we prove the ``if'' direction of Theorem \ref{thm:main} (i.e. Theorem \ref{statement:ours}). This is by far the most challenging (and interesting) part of this paper.
The main step towards proving Theorem \ref{thm:main} is establishing Lemma \ref{lem:reg_main}, which is the key lemma of this paper.
For the reader's convenience, we give in Subsection \ref{subsec:intu} an overview of the key ideas of the proof. As the proofs in Section \ref{sec:preliminary} are somewhat routine, we encourage readers who are familiar with the regularity method to skip Section \ref{sec:preliminary} (at least on their first read), and go directly to Section \ref{sec:positive}.

The ``only if'' direction of Theorem \ref{thm:main} is proved in Section \ref{sec:negative}.
In Section \ref{sec:restricted_models} we prove Theorems 
\ref{theo_rest3}, \ref{theo_rest4}, \ref{theo:rest1} and \ref{theo:rest2}. 
We also
raise two additional problems related to the VDF model; one is to what extent can one extend the results of
Theorems \ref{theo_rest3}-\ref{theo:rest2} 
beyond hereditary properties, and the other asks if the sample complexity in the VDF model is the same as in the standard model (for properties that are testable
in the VDF model), see Subsection \ref{subsec:reduction}. Along the way we resolve another open problem raised in \cite{Goldreich_VDF} (see Lemma
\ref{lem:reduction}).
Throughout the paper, when we say that a function is increasing/decreasing we mean weakly increasing/decreasing (i.e. non-decreasing/non-increasing).



\section{Preliminary Lemmas}\label{sec:preliminary}
In this section we introduce vertex-weighted analogues of some key tools of the regularity method, most notable Szemer\'{e}di's regularity lemma \cite{Szemeredi}, the strong regularity lemma \cite{AFKS}, and the counting lemma, as well as some standard corollaries thereof. We also prove some other auxiliary lemmas needed for the proof of Theorem \nolinebreak \ref{thm:main}.

We start with two simple lemmas regarding probability distributions\footnote{Throughout the paper, we will simply write ``distribution" to mean ``probability distribution".} on a finite set. Given a distribution $\mathcal{D}$ on a set $U$ and a subset $W \subseteq U$, we use the notation $\mathcal{D}(W) := \sum_{w \in W}{\mathcal{D}(w)}$, and call $\mathcal{D}(W)$ the {\em weight} of $W$. We denote by $\mathcal{D}_W$ the distribution $\mathcal{D}$ conditioned on $W$, namely $\mathcal{D}_W(w) = \frac{\mathcal{D}(w)}{\mathcal{D}(W)}$ for every $w \in W$.
\begin{lemma}\label{lem:heavy_cut}
	For every set $U$, for every $\eta \in (0,1)$ and for every distribution $\mathcal{D}$ on $U$, there is a partition $\mathcal{P}$ of $U$ into $\lceil 1/\eta \rceil$ parts such that
	$\sum_{W \in \mathcal{P}}
	{\sum_{\{x,y\} \in \binom{W}{2}}}{\mathcal{D}(x)\mathcal{D}(y)} \leq \eta$.
\end{lemma}
\begin{proof}
	Let $\mathcal{P}$ be a random partition of $U$ into $k := \lceil 1/\eta \rceil$ parts, where each element is assigned to one of the parts uniformly at random and independently of all other elements. Then for every pair of distinct elements $x,y \in U$, the probability that $x$ and $y$ belong to the same part is exactly $\frac{1}{k}$. By linearity of expectation we \nolinebreak have
	$$
	\mathbb{E}\left[ \sum_{W \in \mathcal{P}}
	{\sum_{\{x,y\} \in \binom{W}{2}}}{\mathcal{D}(x)\mathcal{D}(y)} \right] =
	\sum_{\{x,y\} \in \binom{U}{2}}
	{\mathcal{D}(x)\mathcal{D}(y) \cdot \frac{1}{k}}
	< \frac{1}{2} \cdot \frac{1}{k} < \eta,
	$$
	so there is a choice of $\mathcal{P}$ with the required property.
\end{proof}
\begin{lemma}\label{prop:balanced_partition}
	Let $a > 0$ be an integer, let $U$ be a finite set and let
	$\mathcal{D}$ be a distribution on $U$ such that $\mathcal{D}(u) \leq \frac{1}{2a}$ for every $u \in U$. Then there is a partition $U = U_1 \cup \dots \cup U_a$ such that $\mathcal{D}(U_i) \geq \frac{1}{2a}$ for every $1 \leq i \leq a$.
\end{lemma}
\begin{proof}
	We proof is by induction on $a$. The base case $a=1$ is trivial, so we assume from now on that $a \geq 2$. Let $U_1 \subseteq U$ be a set of minimal size satisfying $\mathcal{D}(U_1) \geq \frac{1}{2a}$. Then
	$\mathcal{D}(U_1) \leq \frac{1}{a}$, because otherwise we could remove an arbitrary element of $U_1$ (whose weight by assumption is at most $\frac{1}{2a}$) and thus get a proper subset of $U_1$ having weight at least $\frac{1}{2a}$, in contradiction the minimality of $U_1$. Now set
	$U' := U \setminus U_1$, noting that $\mathcal{D}(U') \geq 1 - \frac{1}{a}$.
	Then every $u \in U'$ satisfies
	$$
	\mathcal{D}_{U'}(u) = \frac{\mathcal{D}(u)}{\mathcal{D}(U')} \leq
	\frac{\frac{1}{2a}}{1 - \frac{1}{a}} = \frac{1}{2(a-1)} \; .
	$$
	So by the induction hypothesis for $(U',\mathcal{D}_{U'})$, there is a partition $U' = U_2 \cup \dots \cup U_a$ such that
	$$
	\mathcal{D}(U_i) = \mathcal{D}_{U'}(U_i) \cdot \mathcal{D}(U') \geq \frac{1}{2(a-1)} \cdot \mathcal{D}(U') \geq
	\frac{1}{2(a-1)} \cdot \left(1 - \frac{1}{a}\right) =
	\frac{1}{2a}
	$$
	for every $2 \leq i \leq a$. This completes the proof.
\end{proof}

We consider vertex-weighted graphs, i.e. pairs $(G,\mathcal{D})$ such that $G$ is a graph and $\mathcal{D}$ is a distribution on $V(G)$.
For a set $X \subseteq V(G)$, the {\em subgraph of $(G,\mathcal{D})$ induced by $X$} is defined to be $(G[X],\mathcal{D}_X)$, where $\mathcal{D}_X$ is the distribution $\mathcal{D}$ conditioned on $X$. The weight of an edge/non-edge $\{x,y\}$ (with respect to $\mathcal{D}$) is defined as $\mathcal{D}(x)\mathcal{D}(y)$. For a pair of disjoint sets $X,Y \subseteq V(G)$ with $\mathcal{D}(X),\mathcal{D}(Y) > 0$, the {\em density} of $(X,Y)$ is denoted by $d(X,Y)$ and defined to be
$d(X,Y) =
\frac{1}{\mathcal{D}(X)\mathcal{D}(Y)}
\sum_{(x,y) \in E(X,Y)}{\mathcal{D}(x)\mathcal{D}(y)}$,
where $E(X,Y)$ is the set of edges with one endpoint in $X$ and one endpoint in $Y$. If $\mathcal{D}(X) = 0$ or $\mathcal{D}(Y) = 0$ then define $d(X,Y) = 0$. 
A pair of disjoint vertex-sets $(X,Y)$ is called {\em $\varepsilon$-regular} if for every $X' \subseteq X$ and $Y' \subseteq Y$ with $\D(X') \geq \varepsilon \D(X)$ and $\D(Y') \geq \varepsilon \D(Y)$, it holds that $|d(X',Y') - d(X,Y)| \leq \varepsilon$. The following lemma describes some basic properties of $\varepsilon$-regular pairs.
\begin{lemma}\label{lem:regular_pair}
	Let $(G,\mathcal{D})$ be a vertex-weighted graph, and let $X,Y \subseteq V(G)$ be disjoint vertex-sets such that $\mathcal{D}(X),\mathcal{D}(Y) > 0$, and such that the pair $(X,Y)$ is $\varepsilon$-regular with density $d$. Then the following holds.
	\begin{enumerate}
		\item For every $\alpha \geq \varepsilon$ and $X' \subseteq X$, $Y' \subseteq Y$ with $\mathcal{D}(X') \geq \alpha \mathcal{D}(X)$ and $\mathcal{D}(Y') \geq \alpha \mathcal{D}(Y)$, the pair $(X',Y')$ has density at least $d - \varepsilon$ and at most $d + \varepsilon$, and is $\varepsilon'$-regular with
		$\varepsilon' = \max\{\varepsilon/\alpha, 2\varepsilon\}$.
		\item The set of vertices $x \in X$ which satisfy $|d(x,Y) - d| > \varepsilon$ has weight less than $2\varepsilon \cdot \mathcal{D}(X)$.
	\end{enumerate}
\end{lemma}

\begin{proof}
	Starting with Item 1, let $X' \subseteq X$ and $Y' \subseteq Y$ be such that $\mathcal{D}(X') \geq \alpha \mathcal{D}(X)$ and $\mathcal{D}(Y') \geq \alpha \mathcal{D}(Y)$. Since $\alpha \geq \varepsilon$, the $\varepsilon$-regularity of $(X,Y)$ implies that $d - \varepsilon \leq d(X',Y') \leq d + \varepsilon$. Now let us show that $(X',Y')$ is $\varepsilon'$-regular with
	$\varepsilon' = \max\{\varepsilon/\alpha, 2\varepsilon\}$. Let $X'' \subseteq X'$ and $Y'' \subseteq Y'$ be such that
	$\mathcal{D}(X'') \geq \varepsilon' \mathcal{D}(X')$ and
	$\mathcal{D}(Y'') \geq \varepsilon' \mathcal{D}(Y')$. Then
	$\mathcal{D}(X'') \geq \frac{\varepsilon}{\alpha} \mathcal{D}(X') \geq \varepsilon \mathcal{D}(X)$ and similarly
	$\mathcal{D}(Y'') \geq \varepsilon \mathcal{D}(Y)$. So by the $\varepsilon$-regularity of $(X,Y)$ we have $|d(X'',Y'') - d(X,Y)| \leq \varepsilon$ and hence $|d(X'',Y'') - d(X',Y')| \leq 2\varepsilon \leq \varepsilon'$, as required.
	
	We now prove Item 2. Let $X^+$ (resp. $X^-$) be the set of all $x \in X$ satisfying $d(x,Y) > d + \varepsilon$ (resp. $d(x,Y) < d - \varepsilon$). We have
	\begin{align*}
	d(X^+,Y) &=
	\frac{1}{\mathcal{D}(X^+)\mathcal{D}(Y)} \cdot
	\sum_{x \in X^+}\sum_{y \in N_Y(x)}{\mathcal{D}(x)\mathcal{D}(y)}
	=
	\frac{1}{\mathcal{D}(X^+)\mathcal{D}(Y)} \cdot
	\sum_{x \in X^+}{\mathcal{D}(x) \cdot \mathcal{D}(Y) \cdot d(x,Y)}
	\\ &>
	\frac{1}{\mathcal{D}(X^+)\mathcal{D}(Y)} \cdot \mathcal{D}(X^+)\mathcal{D}(Y) \cdot (d + \varepsilon) = d + \varepsilon.
	\end{align*}
	So unless $\mathcal{D}(X^+) < \varepsilon \mathcal{D}(X)$, we get a contradiction to the $\varepsilon$-regularity of $(X,Y)$. Similarly, we must have $\mathcal{D}(X^-) < \varepsilon \mathcal{D}(X)$. The assertion follows.
\end{proof}
\noindent
The following is a vertex-weighted counting lemma.
\begin{lemma}[Counting lemma for vertex-weighted graphs]\label{lem:counting}
	For every integer $h \geq 2$ and $\eta \in (0,1)$ there is
	$\delta = \delta_{\ref{lem:counting}}(h,\eta)$ such that the following holds. Let $H$ be a graph on $[h]$ and let $U_1,\dots,U_h$ be pairwise-disjoint vertex-sets in a vertex-weighted graph $(G,\mathcal{D})$, such that the following holds.
	\begin{enumerate}
		\item For every $1 \leq i < j \leq h$, if $\{i,j\} \in E(H)$ then $d(U_i,U_j) \geq \eta$, and if $\{i,j\} \notin E(H)$ then $d(U_i,U_j) \leq 1 - \eta$.
		\item For every $1 \leq i < j \leq h$, the pair $(U_i,U_j)$ is $\delta$-regular.
	\end{enumerate}
	Let $\mathcal{U}$ be the set of all $(u_1,\dots,u_h) \in U_1 \times \dots \times U_h$ such that $u_1,\dots,u_h$ induce a copy of $H$ in which $u_i$ plays the role of $i$ for every $1 \leq i \leq h$. Then
	$\sum_{(u_1,\dots,u_h) \in \mathcal{U}}{\prod_{i=1}^{h}{\mathcal{D}(u_i)}} \geq
	\delta \prod_{i=1}^{h}{\mathcal{D}(U_i)}$.
\end{lemma}
\begin{proof}
	If $\mathcal{D}(U_i) = 0$ for some $1 \leq i \leq h$ then there is nothing to prove, so suppose that $\mathcal{D}(U_i) > 0$ for every $1 \leq i \leq h$. 
	The proof is by induction on $h$. The base case $h=2$ trivially holds with $\delta = \delta(2,\eta) = \eta$. So from now on we assume that $h \geq 3$, and set
	$$
	\delta = \delta(h,\eta) =
	\min\left\{
	\frac{1}{4(h-1)}, \; \frac{\eta}{2}, \; \frac{1}{2} \cdot \left( \frac{\eta}{2} \right)^{h-1} \cdot \delta(h-1,\eta/2) \right\}.
	$$
	For each $2 \leq i \leq h$, let $W_i$ be the set of all vertices $u_1 \in U_1$ for which $|d(u_1,U_i) - d(U_1,U_i)| > \delta$. By Item 2 of Lemma \ref{lem:regular_pair}, we have
	$\mathcal{D}(W_i) < 2\delta \cdot \mathcal{D}(U_1)$. Hence, the set
	$U'_1 := U_1 \setminus \bigcup_{i=2}^{h}{W_i}$ satisfies
	$\mathcal{D}(U'_1) >
	\mathcal{D}(U_1) - (h-1) \cdot 2\delta \cdot \mathcal{D}(U_1) \geq
	\frac{1}{2} \mathcal{D}(U_1)$,
	where in the last inequality we used our choice of $\delta$. Now fix any $u_1 \in U'_1$. We define sets $U'_2,\dots,U'_h$ as follows: for $2 \leq i \leq h$, if $\{1,i\} \in E(H)$ then set
	$U'_i = N_{U_i}(u_1)$, and if $\{1,i\} \notin E(H)$ then set $U'_i = U_i \setminus N_{U_i}(u_1)$. By using Item 1 and the fact that $u_1 \in U'_1$, we get that $\mathcal{D}(U'_i) \geq (\eta - \delta)\mathcal{D}(U_i) \geq \frac{\eta}{2} \cdot \mathcal{D}(U_i)$ for every $2 \leq i \leq h$. By Item 1 of Lemma \ref{lem:regular_pair}, and by Conditions 1-2 of the current lemma, we get that for every $2 \leq i < j \leq h$, the pair $(U'_i,U'_j)$ is $\delta'$-regular with
	$\delta' = 2\delta/\eta \leq \delta(h-1,\eta/2)$, and that if $\{i,j\} \in E(H)$ then $d(U'_i,U'_j) \geq \eta - \delta \geq \eta/2$ and if $\{i,j\} \notin E(H)$ then $d(U'_i,U'_j) \leq 1 - \eta + \delta \leq 1 - \frac{\eta}{2}$.
	
	We now see that the sets $U'_2,\dots,U'_h$ satisfy the requirements of the lemma with respect to the graph $H' = H[\{2,\dots,h\}]$ and with $\frac{\eta}{2}$ in place of $\eta$. Let $\mathcal{U}'$ be the set of all $(u_2,\dots,u_h) \in U'_2 \times \dots \times U'_h$ such that $u_2,\dots,u_h$ induce a copy of $H'$ with $u_i$ playing the role of $i$ for every $2 \leq i \leq h$. By the induction hypothesis, we have
	$$
	\sum_{(u_2,\dots,u_h) \in \mathcal{U}'}{\prod_{i=2}^{h}{\mathcal{D}(u_i)}} \geq \delta(h-1,\eta/2) \cdot \prod_{i=2}^{h}{\mathcal{D}(U'_i)} \geq
	\delta(h-1,\eta/2) \cdot (\eta/2)^{h-1} \cdot \prod_{i=2}^{h}{\mathcal{D}(U_i)} \geq
	2\delta \prod_{i=2}^{h}{\mathcal{D}(U_i)}.
	$$
	For every $(u_2,\dots,u_h) \in \U'$, the tuple $(u_1,\dots,u_h)$ induces a copy of $H$ with $u_i$ playing the role of $i$ for every $1 \leq i \leq h$. Hence, for every $(u_2,\dots,u_h) \in \U'$ we have $(u_1,\dots,u_h) \in \U$ (where $\U$ is defined in the statement of the lemma). Since this is true for every $u_1 \in U'_1$, we get that
	$$
	\sum_{(u_1,\dots,u_h) \in \U}{\prod_{i=1}^{h}{\mathcal{D}(u_i)}} \geq
	\sum_{u_1 \in U'_1}{\mathcal{D}(u_1) \cdot
		2\delta \prod_{i=2}^{h}{\mathcal{D}(U_i)}} =
	\mathcal{D}(U'_1) \cdot 2\delta \prod_{i=2}^{h}{\mathcal{D}(U_i)} \geq
	\delta \prod_{i=1}^{h}{\mathcal{D}(U_i)},
	$$
	as required.
\end{proof}

A partition $\mathcal{P} = \{V_1,\dots,V_r\}$ of the vertex-set of a vertex-weighted graph $(G,\mathcal{D})$ is called {\em $\varepsilon$-regular} if the sum of $\D(V_i)\D(V_j)$ over all pairs $1 \leq i < j \leq r$ for which $(V_i,V_j)$ is not $\varepsilon$-regular, is at most $\varepsilon$.
We now state vertex-weighted versions\footnote{We note that a weighted version of Szemer\'{e}di's regularity lemma, where both vertex-weights and edge-weights are allowed, was proved in \cite{CP}, but only under the assumption that all vertex-weights are $o(1)$. Hence this lemma is unsuitable in our setting.} of Szemer\'{e}di's regularity lemma \cite{Szemeredi} and of the strong regularity lemma \cite{AFKS}.
The proofs of these lemmas appear in the appendix.

\begin{lemma}[Szemer\'{e}di's regularity lemma for vertex-weighted graphs]\label{lem:reg}
	For every $\varepsilon \in (0,1)$ and $m \geq 0$ there is $T = T_{\ref{lem:reg}}(\varepsilon,m)$ such that for every vertex-weighted graph $(G,\mathcal{D})$ and for every partition $\mathcal{P}_0$ of $V(G)$ of size not larger than $m$, there is an $\varepsilon$-regular partition $\mathcal{P}$ of $V(G)$ which has at most $T$ parts and refines $\mathcal{P}_0$.
\end{lemma}

\begin{lemma}[Strong regularity lemma for vertex-weighted graphs]\label{lem:strong_reg}
	For every function $\mathcal{E} : \mathbb{N} \rightarrow (0,1)$ and for every integer $m$, there is $S = S_{\ref{lem:strong_reg}}(\mathcal{E},m)$ such that for every vertex-weighted graph $(G,\mathcal{D})$ and for every partition $\mathcal{P}_0$ of $V(G)$ of size at most $m$, there is a refinement $\mathcal{P}$ of $\mathcal{P}_0$, and a refinement $\mathcal{Q}$ of $\mathcal{P}$, such that the following holds.
	\begin{enumerate}
		\item $|\mathcal{Q}| \leq S$.
		\item The partition $\mathcal{Q}$ is $\mathcal{E}(|\mathcal{P}|)$-regular.
		\item $\sum_{P_1,P_2 \in \mathcal{P}}\sum_{Q_1 \subseteq P_1, Q_2 \subseteq P_2}{\mathcal{D}(Q_1)\mathcal{D}(Q_2) \cdot |d(Q_1,Q_2) - d(P_1,P_2)|} \leq \mathcal{E}(0)$. Here the outer sum is over all unordered pairs of distinct $P_1,P_2 \in \mathcal{P}$, and the inner sum is over all $Q_1,Q_2 \in \mathcal{Q}$ such that $Q_i \subseteq P_i$ for $i = 1,2$.
	\end{enumerate}
\end{lemma}

Our last two lemmas are vertex-weighted analogues of well-known corollaries to Szemer\'edi's regularity lemma and the strong regularity lemma, respectively. The ``unweighted" versions of these corollaries were used in \cite{AS_hereditary} in order to prove that every hereditary property is testable in the standard model.

\begin{lemma}\label{lem:Turan_Ramsey}
	For every integer $t \geq 1$ and for every $\delta > 0$ there is $\zeta = \zeta_{\ref{lem:Turan_Ramsey}}(t,\delta) > 0$, such that the following holds. Let $(G,\mathcal{D})$ be a vertex-weighted graph such that every vertex in $G$ has weight less than $\zeta$. Then there are pairwise-disjoint vertex-sets $Q_1,\dots,Q_t \subseteq V(G)$ with the following properties.
	\begin{enumerate}
		\item $\mathcal{D}(Q_i) \geq \zeta$ for every $1 \leq i \leq t$.
		\item $(Q_i,Q_j)$ is $\delta$-regular for every $1 \leq i < j \leq t$.
		\item Either all pairs $(Q_i,Q_j)$ have density at least $\frac{1}{2}$, or all pairs $(Q_i,Q_j)$ have density less than $\frac{1}{2}$.
	\end{enumerate}
\end{lemma}

\begin{proof}
	Setting
	$a = 4^t$ and
	$\varepsilon = \frac{\delta}{4a^4}$,
	we will prove the lemma with
	$$
	\zeta = \zeta_{\ref{lem:Turan_Ramsey}}(t,\delta) =
	\frac{1}{4a^2 \cdot T_{\ref{lem:reg}}(\varepsilon,a)} \; .
	$$	
	Let $(G,\mathcal{D})$ satisfying $\mathcal{D}(v) < \zeta$ for every $v \in V(G)$. Apply Lemma \ref{prop:balanced_partition} with $U = V(G)$, with the distribution $\mathcal{D}$, and with $a$ as defined above. Lemma \ref{prop:balanced_partition} supplies a partition $V(G) = U_1 \cup \dots \cup U_a$ such that $\mathcal{D}(U_i) \geq \frac{1}{2a}$ for every $1 \leq i \leq a$. Now apply Lemma \ref{lem:reg} to $(G,\mathcal{D})$ with parameter $\varepsilon$ and with the partition $\mathcal{P}_0 := \{U_1,\dots,U_a\}$, to obtain an $\varepsilon$-regular partition $\mathcal{P}$ which refines $\mathcal{P}_0$. For each $1 \leq i \leq a$, put $\mathcal{P}_i = \{P \in \mathcal{P} : P \subseteq U_i\}$, and sample $P_i \in \mathcal{P}_i$ with probability proportional to the weight of the parts, i.e.
	$P_i = P$ with probability $\frac{\mathcal{D}(P)}{\mathcal{D}(U_i)}$ for every $P \in \mathcal{P}_i$. We claim that with positive probability, $\mathcal{D}(P_i) \geq \zeta$ for every $1 \leq i \leq a$, and all pairs $(P_i,P_j)$ are $\delta$-regular. For every $1 \leq i \leq a$, the probability that
	$\mathcal{D}(P_i) < \zeta$ is less than
	$\frac{\zeta \cdot |\mathcal{P}|}{\mathcal{D}(U_i)} \leq
	\frac{\zeta \cdot T_{\ref{lem:reg}}(\varepsilon,a)}{1/2a} \leq \frac{1}{2a}$,
	where in the first inequality we used the guarantees of Lemma \ref{lem:reg}. By the union bound, with probability at least $\frac{1}{2}$ we have $\mathcal{D}(P_i) \geq \zeta$ for every $1 \leq i \leq a$. Next, observe that since $\mathcal{P}$ is $\varepsilon$-regular and as $\varepsilon \leq \delta$, the probability that $(P_i,P_j)$ is not $\delta$-regular (for some specific $1 \leq i < j \leq a$) is at most $\frac{\varepsilon}{\mathcal{D}(U_i)\mathcal{D}(U_j)} \leq
	4a^2 \varepsilon \leq \frac{1}{a^2}$.
	So by taking the union bound over all pairs $1 \leq i < j \leq a$, we get that with probability at least
	$1 - \binom{a}{2} \cdot \frac{1}{a^2} > \frac{1}{2}$, all pairs $(P_i,P_j)$ are $\delta$-regular. This proves our assertion.
	
	We thus showed that there is a choice of $P_1,\dots,P_a$ such that $\mathcal{D}(P_i) \geq \zeta$ for every $1 \leq i \leq a$ and such that $(P_i,P_j)$ is $\delta$-regular for every $1 \leq i < j \leq a$. Now consider an auxiliary graph on $[a]$ in which $\{i,j\}$ is an edge if $d(P_i,P_j) \geq \frac{1}{2}$ and $\{i,j\}$ is a non-edge if $d(P_i,P_j) < \frac{1}{2}$. As $a = 4^t$, a well-known bound on Ramsey numbers implies that this graph contains either a clique or an independent set $\{i_1,\dots,i_t\}$. Then
	$Q_1 = P_{i_1},\dots,Q_t = P_{i_t}$ satisfy the requirements of the lemma.
\end{proof}

\begin{lemma}\label{lem:representatives}
	For every function $\mathcal{E} : \mathbb{N} \rightarrow (0,1)$ and for every integer $m$, there is $S = S_{\ref{lem:representatives}}(\mathcal{E},m) > 0$ such that for every vertex-weighted graph $(G,\mathcal{D})$ and for every partition $\mathcal{P}_0$ of $V(G)$ having size at most $m$, there is a partition $\mathcal{P} = \{P_0,P_1,\dots,P_r\}$ of $V(G)$ and vertex-sets $Q_i \subseteq P_i$ for $1 \leq i \leq r$, such that the following holds:
	\begin{enumerate}
		\item $\mathcal{D}(P_0) < \mathcal{E}(0)$.
		\item For every $1 \leq i \leq r$, $P_i$ is contained in some part of $\mathcal{P}_0$.
		\item $\mathcal{D}(Q_i) \geq 1/S$ for every $1 \leq i \leq r$. In particular, $r \leq S$.
		\item For every $1 \leq i < j \leq r$, the pair $(Q_i,Q_j)$ is $\mathcal{E}(r)$-regular.
		\item $\sum_{1 \leq i < j \leq r}{\mathcal{D}(P_i)\mathcal{D}(P_j) \cdot |d(Q_i,Q_j) - d(P_i,P_j)|} \leq \mathcal{E}(0)$.
	\end{enumerate}
\end{lemma}
\begin{proof}
	We may and will assume $\mathcal{E}$ is monotone decreasing\footnote{Indeed, we can replace $\mathcal{E}$ with
		$\mathcal{E}'(r) = \min_{s \leq r}{\mathcal{E}(s)}$, which is clearly monotone decreasing.}.
	For convenience, put $\varepsilon = \mathcal{E}(0)$.
	Let $\mathcal{E}' : \mathbb{N} \rightarrow (0,1)$ be the function $\mathcal{E}'(r) =
	\min
	\left\{
	\mathcal{E}(r), \frac{\varepsilon^2}{2r^4}, \frac{\varepsilon}{3}
	\right\}$. We will show that one can choose 
	$S = S_{\ref{lem:representatives}}(\mathcal{E},m) := \frac{3s^3}{\varepsilon}$, where $s := S_{\ref{lem:strong_reg}}(\mathcal{E}',m)$.
	Apply Lemma \ref{lem:strong_reg} to $(G,\mathcal{D})$ with parameter $\mathcal{E}'$ and with the given partition $\mathcal{P}_0$, to obtain partitions $\mathcal{P}'$ and $\mathcal{Q}$ such that $\mathcal{P}'$ refines $\mathcal{P}_0$, $\mathcal{Q}$ refines $\mathcal{P}'$, and Items 1-3 in Lemma \ref{lem:strong_reg} hold.
	Let $P_0$ be the union of all parts of $\mathcal{P}'$ of weight less than $\varepsilon/|\mathcal{P}'|$, and let $P_1,\dots,P_r$ be the parts of $\mathcal{P}'$ of weight at least $\varepsilon/|\mathcal{P}'|$. Then we have $\mathcal{D}(P_0) < |\mathcal{P}'| \cdot \varepsilon/|\mathcal{P}'| = \varepsilon$, establishing Item 1.  Now set $\mathcal{P} = \{P_0,P_1,\dots,P_r\}$. It is evident that Item 2 holds.
	
	For each $1 \leq i \leq r$, denote
	$\mathcal{Q}_i = \{Q \in \mathcal{Q} : Q \subseteq P_i\}$, and sample $Q_i \in \mathcal{Q}_i$ with probability proportional to the weight of the parts; in other words, for each $Q \in \mathcal{Q}_i$, the probability that $Q_i = Q$ is $\frac{\mathcal{D}(Q)}{\mathcal{D}(P_i)}$. We will show that with positive probability, $Q_1,\dots,Q_r$ satisfy Items 3-5. For each $1 \leq i \leq r$, the probability that $\mathcal{D}(Q_i) < \frac{\mathcal{D}(P_i)}{3r|\mathcal{Q}|}$ is less than $|\mathcal{Q}| \cdot \frac{1}{3r|\mathcal{Q}|} = \frac{1}{3r}$. By the union bound, the probability that there is $1 \leq i \leq r$ for which
	$\mathcal{D}(Q_i) < \frac{\mathcal{D}(P_i)}{3r|\mathcal{Q}|}$ is less than $\frac{1}{3}$. So with probability larger than $\frac{2}{3}$, for every $1 \leq i \leq r$ we have
	$$
	\mathcal{D}(Q_i) \geq \frac{\mathcal{D}(P_i)}{3r|\mathcal{Q}|} \geq \frac{\varepsilon}{3|\mathcal{P}'|^2|\mathcal{Q}|} \geq \frac{\varepsilon}{3|\mathcal{Q}|^3} \geq
	\frac{\varepsilon}{3s^3} = \frac{1}{S},
	$$
	where the last inequality is due to our choice of $\mathcal{Q}$ via Lemma \ref{lem:strong_reg}.
	
	We now prove that Item 4 holds with probability greater than $\frac{2}{3}$. Fix any $1 \leq i < j \leq r$.
	Since $\mathcal{Q}$ is
	$\varepsilon'$-regular with
	$\varepsilon' = \mathcal{E}'(|\mathcal{P}|') \leq
	\min
	\left\{
	\mathcal{E}(|\mathcal{P}'|), \frac{\varepsilon^2}{2|\mathcal{P}'|^4}
	\right\}$, and since $\mathcal{E}(|\mathcal{P}'|) \leq \mathcal{E}(r)$ (by the monotonicity of $\mathcal{E}$),
	the probability that the pair $(Q_i,Q_j)$ is not $\mathcal{E}(r)$-regular is at most $\frac{\varepsilon^2/(2|\mathcal{P}'|^4)}{\mathcal{D}(P_i)\mathcal{D}(P_j)} \leq \frac{1}{2}|\mathcal{P}'|^{-2} \leq \frac{1}{2}r^{-2}$, where the first inequality holds because $\mathcal{D}(P_i),\mathcal{D}(P_j) \geq \varepsilon/|\mathcal{P}'|$. By the union bound over all pairs $1 \leq i < j \leq r$, the probability that there is $1 \leq i < j \leq r$ for which $(Q_i,Q_j)$ is not $\mathcal{E}(r)$-regular is at most
	$\binom{r}{2} \cdot \frac{1}{2}r^{-2} < \frac{1}{3}$.
	
	It remains to show that Item 5 holds with probability at least $\frac{2}{3}$. Observe that
	\begin{align*}
	&\mathbb{E}\left[
	\sum_{1 \leq i < j \leq r}{\mathcal{D}(P_i)\mathcal{D}(P_j) \cdot |d(Q_i,Q_j) - d(P_i,P_j)|} \right]
	= \\&
	\sum_{1 \leq i < j \leq r}{
		\sum_{Q'_i \in \mathcal{Q}_i, Q'_j \in \mathcal{Q}_j}
		{\mathcal{D}(Q'_i)\mathcal{D}(Q'_j) \cdot \left| d(Q'_i,Q'_j) - d(P_i,P_j) \right|}}
	\leq \frac{\varepsilon}{3},
	\end{align*}
	where in the inequality we used Item 3 of Lemma \ref{lem:strong_reg}, our choice of $\mathcal{E}'$, and the fact that $P_1,\dots,P_r \in \mathcal{P}'$. So by Markov's inequality, the probability that Item 5 fails is at most $\frac{1}{3}$, as required.
\end{proof}

\section{The Main Proof}\label{sec:positive}
In this section we prove the ``if'' direction of Theorem \ref{thm:main}.
In Subsection \ref{subsec:intu} we give a high-level overview of the main obstacle one needs to overcome in proving Theorem \ref{thm:main}, and
the main idea behind the way we overcome it. In Subsection \ref{subsec:main_lemma} we state and prove
Lemma \ref{lem:reg_main}, which constitutes the main ingredient in the proof of Theorem \ref{thm:main}. Finally, we prove (the ``if" direction of) Theorem \ref{thm:main} in Subsection \ref{subsec:main_result}.
%
%

\subsection{Proof overview}\label{subsec:intu}

\paragraph{The main difficulty:} Suppose ${\cal P }$ is an extendable hereditary graph property. We are given a graph $G$ and a distribution ${\cal D}$ so that $G$ is $\varepsilon$-far from ${\cal P}$ with respect to ${\cal D}$. Our goal is to show that a sample of $O(1)$ vertices\footnote{Throughout this subsection, $\Omega(1)$ and $O(1)$ mean positive quantities that depend only on $\varepsilon$ and not on $n$ or ${\cal D}$.} from $G$ finds with high probability (whp) an induced subgraph $F$ of $G$ which does not satisfy
${\cal P}$. There are two ways one can try to tackle this problem. First, one can take a blowup $G'$ of $G$, in which a vertex is replaced by a cluster of vertices whose size is proportional to the vertex's weight under ${\cal D}$, and thus (try to) ``reduce'' the problem to the non-weighted case. While this approach can allow one to handle some properties\footnote{Indeed, this is the approach used in \cite{Goldreich_VDF}.}, it seems that the main bottleneck is that a copy of $F$ in $G'$ does not correspond necessarily to a copy of $F$ in $G$, since $F$ might contain several of the vertices that replaced a vertex of $G$. Moreover, if this vertex $v$ has weight $\Omega(1)$ then even a sample of size $O(1)$ will very likely contain several of the vertices of $G'$  that replaced $v$.

A second approach would be to just reprove the result of \cite{AS_hereditary}, while replacing the regularity lemmas used there with regularity lemmas for vertex-weighted graphs. While such lemmas are indeed not hard to prove (see e.g. Lemmas \ref{lem:counting}-\ref{lem:representatives}), the main problem is again vertices of high weight. Now the issue is that clusters of the regular partition might
contain only a single vertex of high weight, a situation in which one would not be able to embed graphs $F$ that need to use more than one vertex from
the same cluster.

\paragraph{The key new idea:}
The main idea is then to prove a lemma that allows one to partition $G$ into three sets $X,Y,Z$
with the following properties: $(i)$ $Z$ will have total weight at most $\varepsilon/2$, $(ii)$ all vertices in $X$ will have weight at least
$\Omega(1)$, $(iii)$ $Y$ will have a highly regular Szemer\'edi partition, that is, there will be a partition of the vertices of $Y$ into sets $P_1,\ldots,P_r$ so that the bipartite graphs between all pairs $(P_i,P_j)$
are pseudo-random (or regular in the sense of the regularity lemma), $(iv)$ each of the clusters $P_i$ will have ``enough'' vertices, and $(v)$ for each $x \in X$ and set $P_i$, either $x$ will be connected to all vertices of $P_i$ or to none of them. We will now see how a partition with properties $(i)$--$(v)$ can allow one to test $\mathcal{P}$. Let us note that the actual structure we will use is much more complicated than is described in the above five properties (cf. Lemma \ref{lem:reg_main}), and that in the present discussion we intentionally oversimplify some technical aspects in order to highlight our main new idea. For example, we will not actually be able to guarantee that all pairs $(P_i,P_j)$ are pseudo-random (or that the measure of pseudo-randomness of these pairs is sufficient for our purposes); instead, as is common in this type of proofs, we will have ``representative sets" $Q_i \subseteq P_i$ such that all pairs $(Q_i,Q_j)$ are pseudo-random and most have roughly the same density as $(P_i,P_j)$. 

We first claim that $G[X \cup Y]$ (i.e. the graph induced by $X \cup Y$) is $\varepsilon/2$-far from satisfying ${\cal P}$.
Indeed, if this is not the case, then we can first turn the graph induced by these sets into a graph satisfying ${\cal P}$ by making
changes of total weight less than $\varepsilon/2$, and then use the fact that ${\cal P}$ is extendable and the fact that the
total weight of $Z$ is at most $\varepsilon/2$ in order to reconnect the vertices of $Z$ to $X \cup Y$ (and amongst themselves) so that the resulting graph will be in ${\cal P}$. The total weight of edges we thus change is less than $\varepsilon$, a contradiction.

We now examine the partition $P_1,\ldots,P_r$ of $Y$ and perform a ``cleaning'' procedure analogous to the one performed in applications of the
regularity lemma. By this we mean that we make (only!) within $Y$ changes of total weight less than $\varepsilon/2$ so that if after
these changes the set $Y$ contains an induced copy of some (bounded-size) graph $F$, then in the original graph, a sample of $O(1)$ vertices from $Y$ finds one such copy with high probability (whp). Here we will also rely on property $(iv)$ of the partition. The fact that $G[X \cup Y]$ is $\varepsilon/2$-far from satisfying ${\cal P}$ and that we made changes of total weight less than $\varepsilon/2$
when cleaning $Y$, means that $G[X \cup Y]$ (after the cleaning) indeed has an induced copy of a graph $F$ that does not satisfy ${\cal P}$. We now claim that a sample of size
$O(1)$ from $G$ (before the cleaning) finds a copy of $F$ whp. First, since the total weight of $Z$ is small, then sampling from $G$ is (effectively) like sampling from $G[X \cup Y]$. Let now $F_X$ (resp. $F_Y$) be the subgraph of $F$ induced by $X$ (resp. $Y$). By the above discussion, a sample of size
$O(1)$ finds a copy of $F_Y$ whp. Now, and this is the first crucial point, property $(v)$ mentioned above guarantees that the vertices of $X$ which form the copy of
$F_X$, form a copy of $F$ with {\em every} set of vertices in $Y$ which forms a copy of $F_Y$. Now, and this is the second crucial point, property $(ii)$ above guarantees that a sample of $O(1)$ vertices finds {\em the\footnote{By ``the'' we mean that $X$ might contain only a single copy of $F_X$, but this copy has to be of weight $\Omega(1)$. This is in sharp contrast to the situation within $Y$, where each copy of $F_Y$ might have very small weight, but the {\em total} weight of such copies must be $\Omega(1)$.}} copy of $F_X$ contained in $X$ whp. Altogether, the algorithm finds an induced copy of $F$ using $O(1)$ queries.

\paragraph{The new regularity lemma:}
As it turns out, one cannot hope to partition $G$ as described in the first paragraph above, and instead we will have to define a partition
with a much more complicated set of features. This is stated in Lemma \ref{lem:reg_main} in the next subsection. One of the main difficulties
is making sure that parts $P_i$ of the partition of $Y$ will not contain only few (or even a single) vertices of high weight (i.e. we want to guarantee property $(iv)$ stated above). This is done by making sure that the weight of the vertices in $Y$ is very small compared to the weight of the parts $P_1,\dots,P_r$. This in itself is challenging, because at the same time we need to have many parts $P_i$ in order to satisfy property $(v)$ above.
The proof of Lemma \ref{lem:reg_main} will use some of the lemmas of Section \ref{sec:preliminary}, most notably Lemma \ref{lem:representatives}, which we will need to iterate (at least implicitly) in order to find the sought-after partition in the statement of Lemma \ref{lem:reg_main}.

\subsection{The Key Lemma}\label{subsec:main_lemma}

In this subsection we state and prove Lemma \ref{lem:reg_main}, which is the main ingredient in the proof of the ``if'' direction of Theorem \ref{thm:main}.

\begin{lemma}\label{lem:reg_main}
	For every function $\Psi : \mathbb{N} \rightarrow \mathbb{N}$ and $\varepsilon > 0$ there is $S = S_{\ref{lem:reg_main}}(\Psi,\varepsilon) > 0$ such that for every vertex-weighted graph $(G,\mathcal{D})$ there is a partition $V(G) = X \cup Y \cup Z$, a partition $\mathcal{P} = \{P_1,\dots,P_r\}$ of $Y$, vertex-sets $Q_i \subseteq P_i$, and pairwise-disjoint vertex-sets $Q_{i,1}, \dots, Q_{i,t} \subseteq Q_i$, where $t = \Psi(|X|+r)$, such that the following holds:
	\begin{enumerate}
		\item $\mathcal{D}(Z) < \varepsilon$.
		\item Every vertex in $X$ has weight at least $1/S$.
		\item For every $x \in X$ and for every $1 \leq i \leq r$, either $x$ is adjacent to all vertices of $P_i$, or to none of the vertices of $P_i$.
		\item $\sum_{1 \leq i \leq r}\sum_{\{x,y\} \in \binom{P_i}{2}}{\mathcal{D}(x)\mathcal{D}(y)} \leq \varepsilon$.
		\item
		$\sum_{1 \leq i < j \leq r}{
			{\mathcal{D}(P_i)\mathcal{D}(P_j) \cdot |d(Q_i,Q_j) - d(P_i,P_j)|}
		} \leq \varepsilon$.
		\item For every $1 \leq i \leq r$, all pairs $(Q_{i,k},Q_{i,\ell})$ are $\frac{1}{\Psi(|X|+r)}$-regular, and either all pairs $(Q_{i,k},Q_{i,\ell})$ have density at least $\frac{1}{2}$, or all pairs $(Q_{i,k},Q_{i,\ell})$ have density less than $\frac{1}{2}$.
		\item For every $1 \leq i < j \leq r$ and $1 \leq k,\ell \leq t$, the pair $(Q_{i,k},Q_{j,{\ell}})$ is $\frac{1}{\Psi(|X|+r)}$-regular and $|d(Q_{i,k},Q_{j,\ell}) - d(Q_i,Q_j)| \leq \frac{1}{\Psi(|X|+r)}$.
		\item For every $1 \leq i \leq r$ and $1 \leq k \leq t$, $\mathcal{D}(Q_{i,k}) \geq 1/S$.
	\end{enumerate}
\end{lemma}
Note that Items 2 and 8 in Lemma \ref{lem:reg_main} together imply that
$|X| + rt \leq S$.
The following lemma constitutes the main part of the proof of Lemma \ref{lem:reg_main}. After proving Lemma \ref{lem:iterations}, we deduce Lemma \ref{lem:reg_main} from Lemmas \ref{lem:iterations} and \ref{lem:Turan_Ramsey}.

\begin{lemma}\label{lem:iterations}
	For every function $\Psi : \mathbb{N} \rightarrow \mathbb{N}$ and $\varepsilon > 0$ there is $S = S_{\ref{lem:iterations}}(\Psi,\varepsilon) > 0$ such that for every vertex-weighted graph $(G,\mathcal{D})$ there is a partition $V(G) = X \cup Y \cup Z$, a partition $\mathcal{P} = \{P_1,\dots,P_r\}$ of $Y$ and vertex-sets $Q_i \subseteq P_i$ (for $1 \leq i \leq r$) such that Items 1-5 in Lemma \ref{lem:reg_main} hold (with respect to $S = S_{\ref{lem:iterations}}(\Psi,\varepsilon)$), and such that the following two conditions are satisfied.
	\begin{enumerate}[label=(\alph*)]
		\item For every $1 \leq i < j \leq r$, the pair $(Q_i,Q_j)$ is $\frac{1}{\Psi(|X|+r)}$-regular.
		\item For every $1 \leq i \leq r$ the following holds: $\mathcal{D}(Q_i) \geq 1/S$, and all vertices in $Q_i$ have weight less than $\frac{1}{\Psi(|X| + r)} \cdot \mathcal{D}(Q_i)$.
	\end{enumerate}
\end{lemma}
\begin{proof}
	We may and will assume that the function $\Psi$ is monotone increasing\footnote{To guarantee that $\Psi$ is monotone increasing, we can simply replace $\Psi$ with the function $\Psi'(s) := \max\{\Psi(0),\dots,\Psi(s) \}$.},
	and that the function $S_{\ref{lem:representatives}}(\mathcal{E},m)$, whose existence is guaranteed by Lemma \ref{lem:representatives}, is monotone decreasing in $\mathcal{E}$ and monotone increasing in $m$. Here, being monotone decreasing in $\mathcal{E}$ means that if a pair of functions $\mathcal{E}_1,\mathcal{E}_2 : \mathbb{N} \rightarrow (0,1)$ satisfy $\mathcal{E}_1(r) \leq \mathcal{E}_2(r)$ for every $r \in \mathbb{N}$, then $S_{\ref{lem:representatives}}(\mathcal{E}_1,m) \geq S_{\ref{lem:representatives}}(\mathcal{E}_2,m)$ for every $m$.
	For each $s \in \mathbb{N}$, define the function $\mathcal{E}_s : \mathbb{N} \rightarrow (0,1)$ \nolinebreak by
	$$
	\mathcal{E}_s(r) =
	\min\left\{ \frac{\varepsilon}{2}, \frac{1}{\Psi(s+r)}\right\}.
	$$
	Now define the functions $S',S'' : \mathbb{N} \rightarrow \mathbb{N}$ by setting:
	$$
	S'(s) = S_{\ref{lem:representatives}}\left( \mathcal{E}_s, 
	2^{s} \cdot \lceil 1/\varepsilon \rceil \right), \; \;
	S''(s) = \max\left\{s, \; \frac{2S'(s)}{\varepsilon} \cdot
	\Psi\left(s + S'(s) \right)\right\}.
	$$
	Note that $S''(s) \geq s$ for every $s \in \mathbb{N}$, and that $S'$ and $S''$ are monotone increasing.
	We define a monotone increasing sequence $s_1,s_2,\dots$ as follows: $s_1 = 1$, and for each $i \geq 2$, $s_i = S''(s_{i-1})$.
	We will show that the lemma holds with
	$$
	S = S_{\ref{lem:iterations}}(\Psi,\varepsilon) =
	s_{\lceil 2/\varepsilon \rceil}\;.
	$$
	
	Let $(G,\mathcal{D})$ be a vertex-weighted graph. We iteratively define a sequence of pairwise-disjoint vertex-sets ${X_1,X_2,\dots}\subseteq V(G)$ as follows:
	let $X_1$ be the set of all vertices of $G$ of weight at least $1/s_1$; for each $i \geq 2$, let $X_i$ be the set of all vertices in $V(G) \setminus (X_1 \cup \dots \cup X_{i-1})$ having weight at least $1/s_i$.
	Since $X_1,X_2,\dots$ are pairwise-disjoint, there must be 
	$1 \leq i \leq \lceil 2/\varepsilon \rceil$ for which
	$\mathcal{D}(X_i) \leq \varepsilon/2$. We now set $Z' = X_i$, $X = X_1 \cup \dots \cup X_{i-1}$ and
	$Y' = V(G) \setminus (X \cup Z') = V(G) \setminus (X_1 \cup \dots \cup X_i)$. Note that $\mathcal{D}(Z') \leq \varepsilon/2$. Setting $s := s_{i-1} \leq s_{\lceil 2/\varepsilon \rceil - 1} \leq S$, note that every vertex in $X$ has weight at least $\frac{1}{s}$ (so in particular $|X| \leq s$), while every vertex in $Y'$ has weight less than $\frac{1}{s_i} = \frac{1}{S''(s)}$.
	
	If $\mathcal{D}(Y') < \frac{\varepsilon}{2}$ then $\mathcal{D}(Y' \cup Z') < \varepsilon$, so the assertion of the lemma holds for $Y = \emptyset$ and $Z = Z' \cup Y'$, and we are done. So we may and will assume from now on that $\mathcal{D}(Y') \geq \frac{\varepsilon}{2}$.
	Let $\mathcal{P}'_0$ be a partition of $Y'$ into $\lceil 1/\varepsilon \rceil$ parts such that
	$\sum_{P \in \mathcal{P}'_0}{\sum_{\{x,y\} \in \binom{P}{2}}{\mathcal{D}(x)\mathcal{D}(y)}} \leq \varepsilon$, as guaranteed by Lemma \ref{lem:heavy_cut}. For every $x \in X$, consider the partition $\mathcal{P}_x := \{N_{Y'}(x), Y' \setminus N_{Y'}(x)\}$ of $Y'$. Let $\mathcal{P}_0$ be the common refinement of the partitions $\mathcal{P}'_0$ and $(\mathcal{P}_x)_{x \in X}$. Then for every $x \in X$ and $P \in \mathcal{P}_0$, either $x$ is adjacent to every vertex of $P$, or $x$ is not adjacent to any vertex of $P$. Moreover, we have
	$|\mathcal{P}_0| \leq 2^{|X|} \cdot \lceil 1/\varepsilon \rceil \leq
	2^{s} \cdot \lceil 1/\varepsilon \rceil$.
	
	Now apply Lemma \ref{lem:representatives} to $(G[Y'],\mathcal{D}_{Y'})$ with parameters $\mathcal{E}_s$ and $m = 2^{s} \cdot \lceil 1/\varepsilon \rceil$, and with the partition $\mathcal{P}_0$ (noting that $|\mathcal{P}_0| \leq m$), to obtain a partition $\mathcal{P} = \{P_0,P_1,\dots,P_r\}$ of $Y'$ and vertex-sets $Q_i \subseteq P_i$ (for $1 \leq i \leq r$), with the properties stated in that lemma.
	Note that in particular we have
	\begin{equation}\label{eq:r_bound}
	r \leq
	S_{\ref{lem:representatives}}
	\left( \mathcal{E}_s, 2^{s} \cdot \lceil 1/\varepsilon \rceil \right) = S'(s).
	\end{equation}
	Set $Z = Z' \cup P_0$ and $Y = Y' \setminus P_0$, noting that
	$\mathcal{D}(P_0) < \mathcal{E}_s(0) \leq \frac{\varepsilon}{2}$, and hence $\mathcal{D}(Z) = \mathcal{D}(Z') + \mathcal{D}(P_0) < \varepsilon$, as required by Item 1 in Lemma \ref{lem:reg_main}. Items 3 and 4 in Lemma \ref{lem:reg_main} hold because each of the sets $P_1,\dots,P_r$ is contained in some part of $\mathcal{P}_0$, and hence also in some part of $\mathcal{P}'_0$. Item 2 of Lemma \ref{lem:reg_main} was already verified above, and Item 5 of Lemma \ref{lem:reg_main} is guaranteed by Lemma \ref{lem:representatives}. Item (a) holds because Lemma \ref{lem:representatives} guarantees that all pairs $(Q_i,Q_j)$ are $\mathcal{E}_s(r)$-regular, and because $\mathcal{E}_s(r) \leq \frac{1}{\Psi(s+r)} \leq \frac{1}{\Psi(|X|+r)}$ (here we used our choice of $\mathcal{E}_s$, the fact that $|X| \leq s$, and the monotonicity of $\Psi$).
	It remains to prove Item (b). For each $1 \leq i \leq r$, \nolinebreak we \nolinebreak have
	\begin{equation}\label{eq:Q_i bound}
	\begin{split} 
	\mathcal{D}(Q_i) &=
	\mathcal{D}_{Y'}(Q_i) \cdot \mathcal{D}(Y') \geq
	\mathcal{D}_{Y'}(Q_i) \cdot \frac{\varepsilon}{2} \geq
	\frac{\varepsilon}{2S_{\ref{lem:representatives}}
		\left( \mathcal{E}_s, 2^{s} \cdot \lceil 1/\varepsilon \rceil \right)}
	\\&=
	\frac{\varepsilon}{2S'(s)} 	
	\geq \frac{1}{S''(s)}
	\geq
	\frac{1}{S''(s_{\lceil 2/\varepsilon \rceil - 1})} = 
	\frac{1}{s_{\lceil 2/\varepsilon \rceil}} = \frac{1}{S} \; ,
	\end{split}
	\end{equation}
	where in the second inequality we used the guarantees of Lemma \ref{lem:representatives}, and later we used our choice of $S'$ and $S''$, the monotonicity of $S''$, and the fact that $s \leq s_{\lceil 2/\varepsilon \rceil - 1}$.
	Next, fix $1 \leq i \leq r$ and recall that all vertices in
	$Q_i \subseteq Y \subseteq Y'$ have weight less than
	\begin{align*}
	\frac{1}{S''(s)} &\leq
	\frac{1}{\Psi\left( s + S'(s) \right)} \cdot
	\frac{\varepsilon}{2S'(s)}
	\\ &\leq
	\frac{1}{\Psi(s + r)} \cdot \mathcal{D}(Q_i) \leq
	\frac{1}{\Psi(|X| + r)} \cdot \mathcal{D}(Q_i),
	\end{align*}
	where in the first inequality we used our choice of $S''$, in the last two inequalities we used the monotonicity of $\Psi$, and in the second inequality we also used \eqref{eq:r_bound} and an intermediate step in \eqref{eq:Q_i bound}.
	This shows that $\mathcal{D}(u) <
	\frac{1}{\Psi(|X| + r)} \cdot \mathcal{D}(Q_i)$ for every $1 \leq i \leq r$ and $u \in Q_i$, \nolinebreak as \nolinebreak required.
\end{proof}

\begin{proof}[Proof of Lemma \ref{lem:reg_main}]
	Define the functions
	$$\zeta : \mathbb{N} \rightarrow (0,1), \;  \zeta(m) = \zeta_{\ref{lem:Turan_Ramsey}}\left( \Psi(m),\frac{1}{\Psi(m)} \right),
	$$
	and
	$$
	\Psi' : \mathbb{N} \rightarrow \mathbb{N}, \;
	\Psi'(m) = \frac{2\Psi(m)}{\zeta(m)}.
	$$
We may and will assume that the function $\zeta_{\ref{lem:Turan_Ramsey}}(t,\delta)$ is monotone decreasing in $t$ and monotone increasing in $\delta$. This assumption implies that the function $\zeta$ defined above is monotone decreasing. We prove the lemma with
	$$
	S = S_{\ref{lem:reg_main}}(\Psi,\varepsilon) :=
	\frac{S_{\ref{lem:iterations}}(\Psi',\varepsilon)}
	{\zeta(S_{\ref{lem:iterations}}(\Psi',\varepsilon))} \geq S_{\ref{lem:iterations}}(\Psi',\varepsilon)\;.
	$$
	
	Let $(G,\mathcal{D})$ be a vertex-weighted graph.
	Apply Lemma \ref{lem:iterations} to $(G,\mathcal{D})$ with parameters $\Psi'$ and $\varepsilon$, to obtain a partition $V(G) = X \cup Y \cup Z$, a partition $\mathcal{P} = \{P_1,\dots,P_r\}$ of $Y$, and subsets $Q_i \subseteq P_i$ (for $1 \leq i \leq r$) such that Items 1-5 of Lemma \ref{lem:reg_main} hold (with respect to $S_{\ref{lem:iterations}}(\Psi',\varepsilon)$), and so do Items (a) and (b) of Lemma \ref{lem:iterations}.
	
	Let us now prove that Items 6-8 (in Lemma \ref{lem:reg_main}) hold. It will be convenient to put
	$m := |X| + r$. By Item (b) in Lemma \ref{lem:iterations} and by our choice of $\Psi'$, we have
	\begin{equation}\label{eq:vertex_weight_bound}
	\mathcal{D}(u) <
	\frac{1}{\Psi'(m)} \cdot \mathcal{D}(Q_i) <
	\frac{\zeta(m)}{\Psi(m)} \cdot \mathcal{D}(Q_i) \leq
	\zeta(m) \cdot \mathcal{D}(Q_i)
	\end{equation}
	for every $1 \leq i \leq r$ and $u \in Q_i$. Recalling our choice of $\zeta$, we see that Lemma \ref{lem:Turan_Ramsey} is applicable to $(G[Q_i],\mathcal{D}_{Q_i})$ with parameters $t = \Psi(m) = \Psi(|X| + r)$ and $\delta = \frac{1}{\Psi(m)} = \frac{1}{\Psi(|X| + r)}$. Applying Lemma \ref{lem:Turan_Ramsey} with this input, we obtain pairwise-disjoint vertex-sets $Q_{i,1},\dots,Q_{i,t} \subseteq Q_i$ satisfying the properties stated in that lemma. The guarantees of Lemma \ref{lem:Turan_Ramsey} immediately establish Item 6, and also imply that for every
	$1 \leq k \leq t$ we have
	$$
	\mathcal{D}(Q_{i,k}) \geq \zeta(m) \cdot \mathcal{D}(Q_i) =
	\zeta(|X| + r) \cdot \mathcal{D}(Q_i) \geq
	\zeta(|X| + r) \cdot \frac{1}{S_{\ref{lem:iterations}}(\Psi',\varepsilon)} \geq
	\frac{\zeta(S_{\ref{lem:iterations}}(\Psi',\varepsilon))}{S_{\ref{lem:iterations}}(\Psi',\varepsilon)} = \frac{1}{S},
	$$
	where in the second and third inequalities we used the fact that
	$|X|+r,\frac{1}{\mathcal{D}(Q_i)} \leq S_{\ref{lem:iterations}}(\Psi',\varepsilon)$, as guaranteed by Item 2 of Lemma \ref{lem:reg_main} and Item (b) of Lemma \ref{lem:iterations}; in the third inequality we also used the monotonicity of $\zeta$. This establishes Item 8.	
	It remains to prove Item 7.
	By Item (a) of Lemma \ref{lem:iterations}, the pair $(Q_i,Q_j)$ is $\frac{1}{\Psi'(m)}$-regular for every $1 \leq i < j \leq r$. Fix any $1 \leq k,\ell \leq t$. Recalling that
	$\frac{1}{\Psi'(m)} = \frac{\zeta(m)}{2\Psi(m)}$ and that $\mathcal{D}(Q_{i,k}) \geq \zeta(m) \cdot \mathcal{D}(Q_i), \, \mathcal{D}(Q_{j,\ell}) \geq \zeta(m) \cdot \mathcal{D}(Q_j)$, we apply Item 1 of Lemma \ref{lem:regular_pair} to $Q_i,Q_j,Q_{i,k},Q_{j,\ell}$ with parameter $\alpha = \zeta(m)$, to conclude that
	$|d(Q_{i,k},Q_{j,\ell}) - d(Q_i,Q_j)| \leq \frac{1}{\Psi'(m)} \leq \frac{1}{\Psi(m)} = \frac{1}{\Psi(|X|+r)}$, and that the pair $(Q_{i,k},Q_{j,\ell})$ is $\frac{1}{\Psi(|X|+r)}$-regular, as required.
\end{proof}

\subsection{Proof of the Main Result}\label{subsec:main_result}
In this subsection we prove (the ``if" direction of) Theorem \ref{thm:main}.
For a hereditary and extendable graph property $\mathcal{P}$, our tester for $\mathcal{P}$ will work as follows: given an input $(G,\mathcal{D})$ and a proximity parameter $\varepsilon$, the tester samples a sequence of vertices $u_1,\dots,u_s \in V(G)$ independently and with distribution $\mathcal{D}$, where $s = s_{\mathcal{P}}(\varepsilon)$ is as in Theorem \ref{statement:ours}; the tester then accepts if and only if $G[\{u_1,\dots,u_s\}]$ satisfies $\mathcal{P}$. Since $\mathcal{P}$ is hereditary, this tester accepts with probability $1$ if the input graph satisfies $\mathcal{P}$. In the other direction, Theorem \ref{statement:ours} immediately implies that if the input $(G,\mathcal{D})$ is $\varepsilon$-far from $\mathcal{P}$ then the tester rejects with probability at least $\frac{2}{3}$. So we see that the ``if" direction of Theorem \ref{thm:main} follows from Theorem \ref{statement:ours}.

From now on our goal is to prove Theorem \ref{statement:ours}.
We start by introducing variants of some definitions from \cite{AS_hereditary}. An {\em embedding scheme} is a complete graph $K$ with a vertex partition $A_K \cup B_K$, such that every vertex in $B_K$ is colored black or white, every edge with an endpoint in $A_K$ is colored black or white, and every edge contained in $B$ is colored black, white or grey. Note that one of $A_k,B_k$ may be empty; that the vertices of $A_K$ are not colored; and that the edges with at least one endpoint in $A_K$ cannot be colored grey.
An {\em embedding} from a graph $F$ to an embedding scheme $K$ is a map
$\varphi : V(F) \rightarrow V(K)$ such that the following \nolinebreak holds:
\begin{enumerate}
	\item For every $a \in A_K$ we have $|\varphi^{-1}(a)| \leq 1$.
	\item For every $b \in B_K$, if $b$ is colored black then $\varphi^{-1}(b)$ induces a complete graph, and if $b$ is colored white then $\varphi^{-1}(b)$ induces an empty graph.
	\item For every $\{x,y\} \in \binom{V(K)}{2}$, if $\{x,y\}$ is colored black then the bipartite graph between $\varphi^{-1}(x)$ and $\varphi^{-1}(y)$ is complete, and if $\{x,y\}$ is colored white then the bipartite graph between $\varphi^{-1}(x)$ and $\varphi^{-1}(y)$ is empty (note that there are no restrictions in the case that $\{x,y\}$ is colored grey).
\end{enumerate}
Note that Condition 3 implies that for every $a \in A_K$ and $x \in V(K) \setminus \{a\}$, the bipartite graph between $\varphi^{-1}(a)$ and $\varphi^{-1}(x)$ is either complete or empty.
We use the notation $F \rightarrow K$ to mean that there is an embedding from $F$ to $K$.
For a graph-family $\mathcal{F}$ and an integer $m$, let $\mathcal{F}_m$ be the family of all embedding schemes $K$ on at most $m$ vertices, such that there is an embedding from some $F \in \mathcal{F}$ to $K$.
We now introduce a variant of the function $\Psi_{\mathcal{F}}$ defined in \cite{AS_hereditary}.
\begin{definition}\label{def:Psi}
	For a graph-family $\mathcal{F}$ and an integer $m$ for which $\mathcal{F}_m \neq \emptyset$, define
	$$
	\Psi_{\mathcal{F}}(m) =
	\max_{K \in \mathcal{F}_m}{
		\min_{F \in \mathcal{F} : F \rightarrow K}{|V(F)|}
	}.
	$$
	If $\mathcal{F}_m = \emptyset$ then define $\Psi_{\mathcal{F}}(m) = 0$.
\end{definition}
\noindent
We are now ready to prove Theorem \ref{statement:ours} (and thus also the ``if" direction of Theorem \ref{thm:main}).

\begin{proof}[Proof of Theorem \ref{statement:ours}]
	Let $\mathcal{P}$ be a hereditary and extendable graph property. 
	Let $\mathcal{F} = \mathcal{F}(\mathcal{P})$ be the family of graphs which do not satisfy $\mathcal{P}$. Fix $\varepsilon \in (0,1)$, and let
	$\Psi : \mathbb{N} \rightarrow \mathbb{N}$ be the function
	$$
	\Psi(m) =
	\max\left\{
	\frac{8}{\varepsilon},
	\Psi_{\mathcal{F}}(m), \frac{1}{\delta_{\ref{lem:counting}}(\Psi_{\mathcal{F}}(m),\frac{\varepsilon}{8})}
	\right\},
	$$
	where $\Psi_{\mathcal{F}}$ is defined in Definition \ref{def:Psi}.
	We may and will assume that the function $\delta_{\ref{lem:counting}}(h,\eta)$ is monotone decreasing in $h$ and monotone increasing in $\eta$. Set
	$S := S_{\ref{lem:reg_main}}(\Psi, \frac{\varepsilon}{4})$. We prove the theorem with
	\begin{equation}\label{eq:choice_of_s}
	s = s_{\mathcal{P}}(\varepsilon) := \frac{2S^{S+1}}{\delta_{\ref{lem:counting}}\left( S,\frac{\varepsilon}{8} \right)} \; .
	\end{equation}
	
	Let $(G,\mathcal{D})$ be a vertex-weighted graph which is $\varepsilon$-far from $\mathcal{P}$. Apply Lemma \ref{lem:reg_main} to $(G,\mathcal{D})$ with parameter $\frac{\varepsilon}{4}$ and with $\Psi$ as above, to obtain a partition
	$V(G) = X \cup Y \cup Z$, a partition $\{P_1,\dots,P_r\}$ of $Y$, subsets $Q_i \subseteq P_i$ (for $1 \leq i \leq r$), and pairwise-disjoint subsets $Q_{i,1},\dots,Q_{i,t} \subseteq Q_i$, such that $t = \Psi(|X| + r)$ and Items 1-8 in Lemma \ref{lem:reg_main} hold.
	
	We claim that $G$ is $\frac{3\varepsilon}{4}$-far from any graph $G'$ on $V(G)$ which satisfies $G'[X \cup Y] \in \mathcal{P}$.
	So suppose by contradiction that there is a graph $G'$ on $V(G)$ such that $G'[X \cup Y]$ satisfies $\mathcal{P}$ and such that $G'$ is $\frac{3\varepsilon}{4}$-close to $G$. Since $\mathcal{P}$ is extendable,
	there is a graph $G''$ on $V(G) = V(G')$ such that
	$G''[X \cup Y] = G'[X \cup Y]$ and such that $G''$ satisfies $\mathcal{P}$. In order to turn $G'$ into $G''$, we only need to add/delete edges which are incident to vertices of $Z$. Therefore, the total weight of edge-changes needed to turn $G'$ into $G''$ is at most $\mathcal{D}(Z) < \frac{\varepsilon}{4}$, as guaranteed by Item 1 of Lemma \ref{lem:reg_main}. So we see that $G$ can be turned into $G''$, which satisfies $\mathcal{P}$, by adding/deleting edges whose total weight is less than $\frac{3\varepsilon}{4} + \frac{\varepsilon}{4} = \varepsilon$, in contradiction the assumption that $(G,\mathcal{D})$ is $\varepsilon$-far from \nolinebreak $\mathcal{P}$.
	
	We thus proved that $G$ is $\frac{3\varepsilon}{4}$-far from any graph $G'$ satisfying $G'[X \cup Y] \in \mathcal{P}$.
	Now, let $G'$ be the graph obtained from $G$ by doing the following changes:
	\begin{enumerate}
		\item For every $1 \leq i \leq r$, if $d(Q_{i,k},Q_{i,\ell}) \geq \frac{1}{2}$ for every $1 \leq k < \ell \leq t$ then turn $P_i$ into a clique, and if $d(Q_{i,k},Q_{i,\ell}) < \frac{1}{2}$ for every $1 \leq k < \ell \leq t$, then turn $P_i$ into an independent set. By Item 6 in Lemma \ref{lem:reg_main}, one of these options has to hold. The total weight of edge-changes
		needed in this item is at most $\frac{\varepsilon}{4}$ by Item 4 of Lemma  \ref{lem:reg_main}.
		\item For every $1 \leq i < j \leq r$, if
		$d(Q_i,Q_j) > 1 - \frac{\varepsilon}{4}$ then add all edges between $P_i$ and $P_j$, and if $d(Q_i,Q_j) < \frac{\varepsilon}{4}$ then remove all edges between $P_i$ and $P_j$ (note that if
		$\frac{\varepsilon}{4} \leq d(Q_i,Q_j) \leq 1 - \frac{\varepsilon}{4}$ then no changes are made in the bipartite graph between $P_i$ and $P_j$). The total weight of edge-changes needed in this item is less than $\frac{\varepsilon}{2}$ by Item 5 of Lemma \ref{lem:reg_main}. Indeed, observe that the total weight of changes between $P_i,P_j$ is less than $\mathcal{D}(P_i)\mathcal{D}(P_j) \cdot \left( |d(Q_i,Q_j) - d(P_i,P_j)| + \frac{\varepsilon}{4} \right)$ by the triangle inequality. Hence, the total weight of changes is less than
		\begin{align*}
		&\sum_{1 \leq i < j \leq r}
		{\mathcal{D}(P_i)\mathcal{D}(P_j) \cdot \left( |d(Q_i,Q_j) - d(P_i,P_j)| + \frac{\varepsilon}{4} \right)} \leq \\& 
		\frac{\varepsilon}{4} + \sum_{1 \leq i < j \leq r}
		{\mathcal{D}(P_i)\mathcal{D}(P_j) \cdot |d(Q_i,Q_j) - d(P_i,P_j)|} \leq \frac{\varepsilon}{2}.
		\end{align*}
	\end{enumerate}
	Note that no edge with an endpoint in $X$ was added/deleted in Items 1-2, so $G'$ and $G$ agree on all edges that are incident to vertices of $X$.
	
	We see that the total weight of edge-changes made in Items 1-2 is less than $\frac{3\varepsilon}{4}$. So $G'[X \cup Y]$ cannot satisfy $\mathcal{P}$, implying that $G'[X \cup Y] \in \mathcal{F}$. 
	Note that by definition (see Items 1-2 above), the graph $G'$ has the following properties: 
	\begin{enumerate}
		\item[(a)] For every $1 \leq i \leq r$, $P_i$ is either a clique or an independent set in $G'$. Moreover, $P_i$ is a clique in $G'$ then
		$d_G(Q_{i,k},Q_{i,\ell}) \geq \frac{1}{2}$ for every $1 \leq k < \ell \leq t$, and if $P_i$ is an independent set in $G'$ then
		$d_G(Q_{i,k},Q_{i,\ell}) < \nolinebreak \frac{1}{2}$ for every $1 \leq k < \ell \leq t$. 
		\item[(b)] For every pair $1 \leq i < j \leq r$, if there is an edge in $G'$ between $P_i$ and $P_j$ then $d_G(Q_i,Q_j) \geq \frac{\varepsilon}{4}$. Then by Item 7 of Lemma \ref{lem:reg_main} we have that
		$d_G(Q_{i,k},Q_{j,\ell}) \geq \frac{\varepsilon}{4} - \frac{1}{\Psi(|X| + r)} \geq \frac{\varepsilon}{8}$ for every $1 \leq k,\ell \leq t$.
		Analogously, if there is a non-edge in $G'$ between $P_i$ and $P_j$ then $d_G(Q_i,Q_j) \leq 1 - \frac{\varepsilon}{4}$, which implies (by Item 7 of Lemma \ref{lem:reg_main}) that
		$d_G(Q_{i,k},Q_{j,\ell}) \leq 1 - \frac{\varepsilon}{4} + \frac{1}{\Psi(|X| + r)} \leq 1 - \frac{\varepsilon}{8}$ for every $1 \leq k,\ell \leq t$.
	\end{enumerate}
	
	Now let $K$ be the following embedding scheme: $A_K = X$ and $B_K = \{b_1,\dots,b_r\}$;
	for each $1 \leq i \leq r$, vertex $b_i$ is colored black if $P_i$ is a clique in $G'$ and white if $P_i$ is an independent set in $G'$; for each $x,x' \in X$, edge $\{x,x'\}$ is colored black if $\{x,x'\} \in E(G)$ and white if $\{x,x'\} \notin E(G)$;
	for each $x \in X$, $1 \leq i \leq r$, edge $\{x,b_i\}$ is colored black if the bipartite graph between $x$ and $P_i$ is complete and white if this bipartite graph is empty (Item 3 in Lemma \ref{lem:reg_main} implies that one of these options must hold); finally, for every $1 \leq i < j \leq r$, edge $\{b_i,b_j\}$ is colored black if the bipartite graph between $P_i$ and $P_j$ is complete in $G'$, white if the bipartite graph between $P_i$ and $P_j$ is empty in $G'$, and grey otherwise.

	Observe that the map $\varphi : X \cup Y \rightarrow V(K)$ which maps $x$ to itself (for every $x \in X = A_K$) and $P_i$ to $b_i$ (for every $1 \leq i \leq r$), is an embedding from $G'[X \cup Y]$ to $K$. Since $|V(K)| = |X| + r$, we have
	$K \in \mathcal{F}_m$ for $m := |X| + r$. By the definition of the function $\Psi_{\mathcal{F}}$ (see Definition \ref{def:Psi}), there is
	$F \in \mathcal{F}$ such that $F \rightarrow K$ and
	$|V(F)| \leq \Psi_{\mathcal{F}}(m) = \Psi_{\mathcal{F}}(|X| + r) \leq \Psi(|X| + r) = t$.
	
	Now, fixing an embedding $\rho$ from $F$ to $K$, write $W_i := \rho^{-1}(b_i) = \{w_{i,1},\dots,w_{i,f_i}\}$ for $1 \leq i \leq r$.
	Put $W = W_1 \cup \dots \cup W_r$ and $H = F[W]$.
	We claim that the sets $(Q_{i,k})_{1 \leq i \leq r, 1 \leq k \leq f_i}$ satisfy the requirements 1-2 in Lemma \ref{lem:counting} with respect to $h = |V(F)| \leq \Psi_{\mathcal{F}}(m)$, $\eta = \frac{\varepsilon}{8}$ and $H$ as above, {\em in the graph $G$}. In other words, we show that one can apply Lemma \ref{lem:counting} with the sets $U_1,\dots,U_h$ being
	$(Q_{i,k})_{1 \leq i \leq r, 1 \leq k \leq f_i}$, and with $G$ as the host graph. We actually already proved that Item 1 in Lemma \ref{lem:counting} holds; indeed, this follows from the fact that $F \rightarrow K$, the definition of the embedding scheme $K$, and Items (a)-(b) above. Item 2 of Lemma \ref{lem:counting} follows from Items 6-7 of Lemma \ref{lem:reg_main}, which together imply that for every $1 \leq i \leq j \leq r$ and $1 \leq k \leq f_i, 1 \leq \ell \leq f_j$ (with the exception of $(i,k)$ = $(j,\ell)$), the pair $(Q_{i,k},Q_{j,\ell})$ is $\delta$-regular with
	$\delta = \frac{1}{\Psi(m)} \leq
	\delta_{\ref{lem:counting}}(\Psi_{\mathcal{F}}(m),\frac{\varepsilon}{8}) \leq \delta_{\ref{lem:counting}}(h,\frac{\varepsilon}{8})$, as required.
	
	We thus showed that Lemma \ref{lem:counting} is applicable to the tuple of sets  $(Q_{i,k})_{1 \leq i \leq r, 1 \leq k \leq f_i}$ and the graph $H = F[W]$ (with the parameters defined above). Let $\mathcal{U}$ be the set of all tuples
	$(u_{i,k})_{1 \leq i \leq r, 1 \leq k \leq f_i}$, where $u_{i,k} \in Q_{i,k}$, which induce (in $G$) a copy of $H = F[W]$ in which $u_{i,k}$ plays the role of $w_{i,k}$ for every $1 \leq i \leq r$ and $1 \leq k \leq f_i$. By Lemma \ref{lem:counting}, we have
	\begin{equation}\label{eq:F'_copy_prob_bound}
	\sum_{(u_{i,k})_{i,k} \in \mathcal{U}}
	\; {\prod_{i=1}^{r}\prod_{k=1}^{f_i}{\mathcal{D}(u_{i,k})}} \geq
	\delta_{\ref{lem:counting}}\left( h,\frac{\varepsilon}{8} \right) \cdot  \prod_{i=1}^{r}\prod_{k=1}^{f_i}{\mathcal{D}(U_{i,k})} \geq
	\delta_{\ref{lem:counting}}\left( \Psi_{\mathcal{F}}(m),\frac{\varepsilon}{8} \right) \cdot S^{-|W|},
	\end{equation}
	where in the last inequality we used the guarantees of Item 8 in Lemma \ref{lem:reg_main} and the monotonicity of the function $\delta_{\ref{lem:counting}}$. Observe that for every $(u_{i,k})_{i,k} \in \mathcal{U}$, the subgraph of $G$ induced by the vertex-set \linebreak 
	$X \cup \{u_{i,k} : 1 \leq i \leq r, 1 \leq k \leq f_i\}$ contains an induced copy of $F$. Indeed, this follows from the definition of $\mathcal{U}$, the fact that $F \rightarrow K$, and the definition of the embedding scheme $K$.
	Now sample an $(|X| + |W|)$-tuple of vertices from $G$ according to the distribution $\mathcal{D}$ and independently.
	Note that if every vertex in $X$ appears in the first $|X|$ vertices of the sample, and if the tuple of the last $|W|$ vertices of the sample belongs to $\mathcal{U}$, then the subgraph induced by the sample contains an induced copy of $F$ and hence does not satisfy $\mathcal{P}$ (as $F \in \mathcal{F}$). 
	The probability for this event is at least
	\begin{align*}
	\delta_{\ref{lem:counting}}\left( \Psi_{\mathcal{F}}(m),\frac{\varepsilon}{8} \right) \cdot S^{-|X|-|W|} \; .
	\end{align*}
	Here we used \eqref{eq:F'_copy_prob_bound} and Item 2 in Lemma \ref{lem:reg_main}. Next, note that $|X| + |W| \leq |X| + rt \leq S$, where in the last inequality we used Items 2 and 8 of Lemma \ref{lem:reg_main}. Similarly, $\Psi_{\mathcal{F}}(m) \leq t \leq S$. So we see that a sample of $S$ random vertices induces a graph which does not satisfy $\mathcal{P}$ with probability at least
	$\delta_{\ref{lem:counting}}\left( S,\frac{\varepsilon}{8} \right) \cdot S^{-S}$.
	Therefore, a sample of $s = s_{\mathcal{P}}(\varepsilon)$ vertices (see \eqref{eq:choice_of_s})
	induces a graph not satisfying $\mathcal{P}$ with probability at least
	$$
	1 - \left( 1 - \delta_{\ref{lem:counting}}\left( S,\frac{\varepsilon}{8} \right) \cdot S^{-S} \right)^{s/S} = 
	1 - \left( 1 - \delta_{\ref{lem:counting}}\left( S,\frac{\varepsilon}{8} \right) \cdot S^{-S} \right)^{\frac{2S^{S}}
	{\delta_{\ref{lem:counting}}\left( S,\frac{\varepsilon}{8} \right)}}
	\geq
	1 - e^{-2} \geq \frac{2}{3} ,
	$$
	as required. This completes the proof.
\end{proof}
It is natural to ask about the dependence on $\varepsilon$ of the sample complexity of the tester supplied by Theorem \nolinebreak \ref{thm:main}. One answer is that one cannot prove any upper bound on the sample complexity which holds uniformly for all properties $\mathcal{P}$, because it was shown in \cite{AS_mono} that no such bound exists even in the standard model. Suppose then that
one is interested only in ``simple'' properties such as induced $H$-freeness (for some fixed $H$). In this case, it is not too hard
to see that although we are iterating Lemma \ref{lem:representatives}, which has wowzer-type (that is, iterated-tower) bounds\footnote{To be precise, we mean here that the ``standard" way of establishing Lemma \ref{lem:representatives} (which is also the way we prove this lemma in this paper) is via the strong regularity lemma (see Lemma \ref{lem:strong_reg}), which is known to only give wowzer-type bounds \cite{CF,KS}. In \cite{CF}, (an unweighted variant of) Lemma \ref{lem:representatives} was proved without the use of the strong regularity lemma, thus giving better, tower-type, bounds. This is alluded to in the following sentence.} in this setting
even for unweighted graphs (see \cite{CF,KS}), we are still getting ``only'' a wowzer-type bound.
We should also point out that it might be possible to use the
ideas in \cite{CF}, together with those presented here, in order to get tower-type bounds on the sample complexity of testing induced $H$-freeness in the VDF model.
\section{VDF-Testable Properties are Extendable and Hereditary}\label{sec:negative}
In this section we prove the ``only if'' direction of Theorem \ref{thm:main}. The proof is divided between Propositions \ref{prop:hard_extendable} and \ref{prop:hard_hereditary}. As shown in \cite{Goldreich_VDF}, we can (and will) always assume that a VDF tester only queries the input graph on pairs of vertices which it has sampled.
\begin{proposition}\label{prop:hard_extendable}
	If a graph property $\mathcal{P}$ is not extendable, then $\mathcal{P}$ is not testable in the VDF model.
\end{proposition}
\begin{proof}
	Since $\mathcal{P}$ is not extendable, there is a graph $G_1 \in \mathcal{P}$, such that no $(|V(G_1)|+1)$-vertex graph satisfying $\mathcal{P}$ contains $G_1$ as an induced subgraph. Let
	$G_2$ be a graph obtained from $G_1$ by adding a ``new'' vertex $v$
	(and putting an arbitrary bipartite graph between $v$ and $V(G_1)$),
	let $\mathcal{D}_1$ be the uniform distribution on $V(G_1)$,
	and let $\mathcal{D}_2$ be the distribution on $V(G_2)$ which assigns weight $\frac{1}{|V(G_1)|}$ to each $u \in V(G_1) \subseteq V(G_2)$ and weight\footnote{Evidently, if one does not wish to allow vertices of weight $0$, then one can instead assign to $v$ a weight tending to $0$; or, more accurately, a weight that is small enough with respect to (the inverse of) the sample complexity of an alleged tester for $\mathcal{P}$ (in a proof by contradiction that such a tester does not exist).} $0$ to $v$.
	
	It is clear that for every integer $q$, a sample of $q$ vertices from $G_1$ according to $\mathcal{D}_1$ is indistinguishable from a sample of $q$ vertices from $G_2$ according to $\mathcal{D}_2$. Observe that $G_1$ satisfies $\mathcal{P}$ while $(G_2,\mathcal{D}_2)$ is $\frac{1}{|V(G_1)|^2}$-far from $\mathcal{P}$. To see that the latter statement is true, observe that by our choice of $G_1$, no matter how we change the bipartite graph between $v$ and $V(G_1)$, we will always get a graph that does not satisfy $\mathcal{P}$. Hence, in order to make $G_2$ satisfy $\mathcal{P}$, one must change the adjacency relation between a pair of vertices from $V(G_1)$, whose weight (under $\mathcal{D}_2$) is $\frac{1}{|V(G_1)|}$.
	
	Now, the fact that $(G_1,\mathcal{D}_1)$ and $(G_2,\mathcal{D}_2)$ are indistinguishable implies that $\mathcal{P}$ is not testable\footnote{\label{footnote:hereditary_non_extendable}We note that if $\mathcal{P}$ is non-extendable but hereditary, then one can easily obtain infinitely many examples showing that $\mathcal{P}$ is not testable (rather than just the one example given in the proof of Proposition \ref{prop:hard_extendable}). Indeed, instead of adding just one vertex to $G_1$, one can add to $G_1$ any number $k$ of vertices (for a large $k$), and give these new vertices weight $o(1/k)$, while distributing the remaining weight uniformly among the vertices of $G_1$ (note that such an assignment is precisely what the setting of Theorem \ref{theo:rest2} forbids). The assumption that $\mathcal{P}$ is hereditary implies that every graph obtained in this way is $\frac{1-o(1)}{|V(G_1)|^2}$-far from satisfying \nolinebreak $\mathcal{P}$. Also, if the weight given to the ``new'' vertices is small enough, then these two weighted graphs are indistinguishable by a sample of any prescribed size.} in the VDF \nolinebreak model.
\end{proof}

\begin{proposition}\label{prop:hard_hereditary}
	If a graph property $\mathcal{P}$ is not hereditary, then $\mathcal{P}$ is not testable in the VDF model.
\end{proposition}
\begin{proof}
	Since $\mathcal{P}$ is not hereditary, there is a graph $G_1$ and an induced subgraph $G_2$ of $G_1$, such that $G_1$ satisfies $\mathcal{P}$ but $G_2$ does not. Let $\mathcal{D}_2$ be the uniform distribution on $V(G_2)$, and let $\mathcal{D}_1$ be the distribution on $V(G_1)$ which is supported on $V(G_2) \subseteq V(G_1)$ and uniform when conditioned on $V(G_2)$, i.e.
	$\mathcal{D}_1(u) = \frac{1}{|V(G_2)|}$ if $u \in V(G_2)$ and $\mathcal{D}_1(u) = 0$ if $u \in V(G_1) \setminus V(G_2)$.
	Clearly, for every integer $q$, a sample of $q$ vertices from $G_1$ according to $\mathcal{D}_1$ is indistinguishable from a sample of $q$ vertices from $G_2$ according to $\mathcal{D}_2$. Also, $G_1$ satisfies $\mathcal{P}$, whereas $(G_2,\mathcal{D}_2)$ is $\frac{1}{|V(G_2)|^2}$-far from $\mathcal{P}$ because $G_2 \notin \mathcal{P}$. Thus, $\mathcal{P}$ is not testable\footnote{In analogy to Footnote \ref{footnote:hereditary_non_extendable}, we note that if $\mathcal{P}$ is non-hereditary but extendable then one can obtain infinitely many examples showing that $\mathcal{P}$ is not testable (rather than just the one given in the proof of Proposition \ref{prop:hard_hereditary}).
		Indeed, the extendability of $\mathcal{P}$ implies that there are arbitrarily large graphs which satisfy $\mathcal{P}$ and contain $G_1$ (and hence also $G_2$) as an induced subgraph. Each of these graphs (together with an appropriate distribution, as in the proof of Proposition \ref{prop:hard_hereditary}) is a witness to the non-testability \nolinebreak of \nolinebreak $\mathcal{P}$.} in the VDF model.
\end{proof}

\section{On Variations of the VDF Model and Related Problems}\label{sec:restricted_models}
In the following two subsections we prove Theorems \ref{theo_rest3}, \ref{theo_rest4}, \ref{theo:rest1} and \ref{theo:rest2}. We then consider two additional problems related to the VDF model; one problem asks if the query complexity in the VDF model is the same as in the standard model (for $\mathcal{P}$ that are testable
in the VDF model), and the other asks for a characterization of the properties that are testable in variants of the VDF model
(as in Theorems \nolinebreak \ref{theo_rest3}-\ref{theo:rest2}).
We start by giving the precise definitions of the settings considered in Theorems \ref{theo_rest3}-\ref{theo:rest2}. 

\paragraph{The ``large inputs" model}
In this model, a property $\mathcal{P}$ is testable if there exists a function $M_{\mathcal{P}} : (0,1) \rightarrow \mathbb{N}$ such that for every
$\varepsilon > 0$, $\mathcal{P}$ is $\varepsilon$-testable with sample complexity depending only on $\varepsilon$ under the promise that inputs $(G,\mathcal{D})$ always satisfy $|V(G)| \geq \nolinebreak M_{\mathcal{P}}(\varepsilon)$.

\paragraph{The ``size-aware" model} In this model, testers are allowed to receive, as part of the input, the number of vertices of the input graph. 

\paragraph{The ``no heavy-weights" (NHW) model} In this model, a property $\mathcal{P}$ is testable if there exists a function $c_{\mathcal{P}} : (0,1) \rightarrow (0,1)$ such that for every $\varepsilon > 0$, $\mathcal{P}$ is $\varepsilon$-testable with sample complexity depending only on $\varepsilon$ under the promise that inputs $(G,\mathcal{D})$ always satisfy
$\max_{v \in V(G)}{\mathcal{D}(v)} \leq c_{\mathcal{P}}(\varepsilon)$.

\paragraph{The ``no light-weights" (NLW) model} In this model, a property $\mathcal{P}$ is testable if for all $\varepsilon, \delta > 0$, $\mathcal{P}$ is $\varepsilon$-testable with sample complexity depending only on $\varepsilon$ and $\delta$ under the promise that inputs $(G,\mathcal{D})$ always satisfy
$\min_{v \in V(G)}{\mathcal{D}(v)} \geq \delta/|V(G)|$.

\vspace{0.3cm}
Theorem \ref{theo_rest3} (resp. \ref{theo_rest4}, \ref{theo:rest1}, \ref{theo:rest2}) then states that every hereditary property is testable in the ``large inputs" (resp. ``size-aware", NHW, NLW) model\footnote{Note that if $\mathcal{P}$ is testable in the ``large inputs" model then it is also testable in the NHW model, because by setting $c_{\mathcal{P}}(\varepsilon) := 1/M_{\mathcal{P}}(\varepsilon)$ we can make sure that the input graph has at least $M_{\mathcal{P}}(\varepsilon)$ vertices. Still, we decided to include a separate proof for Theorem \ref{theo:rest1} (instead of deducing it from Theorem \ref{theo_rest3}) for two reasons: one is that in the course of the proof we resolve another open question raised in \cite{Goldreich_VDF}; and the other is that our proof of Theorem \ref{theo:rest1} shows that $\mathcal{P}$ is testable (in the NHW model) by a tester that accepts if and only if the subgraph induced by the sample satisfies $\mathcal{P}$, whereas the tester given by the proof of Theorem \ref{theo_rest3} is not always of this form.}.

\subsection{Proof of Theorems \ref{theo_rest3}, \ref{theo_rest4} and \ref{theo:rest2}}
\label{subsec:no_low_weight_vertices}
In this subsection we prove Theorems \ref{theo_rest3}, \ref{theo_rest4} and \ref{theo:rest2}, i.e. we show that every hereditary property is testable (with one-sided error) in the ``large inputs", ``size-aware" and NLW models.
Let us introduce some definitions that we will use throughout this subsection. Let $\mathcal{P}$ be a hereditary graph property. 
A graph $F$ is called {\em $\mathcal{P}$-good} if for every $r \geq |V(F)|$ there is an $r$-vertex graph which satisfies $\mathcal{P}$ and contains $F$ as an induced subgraph; this in particular implies that $F$ itself satisfies $\mathcal{P}$. If $F$ is not $\mathcal{P}$-good then it is called 
{\em ${\mathcal{P}}$-bad}, and we denote by $r_{\mathcal{P}}(F)$ the minimal $r \geq |V(F)|$ such that there is no $r$-vertex graph which satisfies $\mathcal{P}$ and contains $F$ as an induced subgraph. In particular, if $F$ does not satisfy $\mathcal{P}$ then it is $\mathcal{P}$-bad and $r_{\mathcal{P}}(F) = |V(F)|$. 
Note that since $\mathcal{P}$ is hereditary, if $F$ is $\mathcal{P}$-bad then there is no graph on $r$ vertices for any $r \geq r_{\mathcal{P}}(F)$ which satisfies $\mathcal{P}$ and contains $F$ as an induced subgraph.
Now let $\mathcal{H} = \mathcal{H}(\mathcal{P})$ be the property of being $\mathcal{P}$-good. Then $\mathcal{H} \subseteq \mathcal{P}$ and $\mathcal{H}$ is hereditary, which follows from the definition of $\mathcal{P}$-goodness and the fact that $\mathcal{P}$ is hereditary.
Observe moreover that $\mathcal{H}$ is extendable.
Indeed, let $G \in \mathcal{H}$, and suppose, for the sake of contradiction, that for every $G'$ on $|V(G)|+1$ vertices which contains $G$ as an induced subgraph, it holds that $G' \notin \mathcal{H}$. Then for every such $G'$, there is no graph on $r_{\mathcal{P}}(G')$ vertices that satisfies $\mathcal{P}$ and contains $G'$ as an induced subgraph. But this means that there is no graph on $\max_{G'}{r_{\mathcal{P}}(G')}$ vertices which satisfies $\mathcal{P}$ and contains $G$ as an induced subgraph, in contradiction to $G \in \mathcal{H}$. 
We note also that if $\mathcal{P}$ itself is extendable then $\mathcal{H} = \mathcal{P}$. 

For an integer $s \geq 1$, let $R_{\mathcal{P}}(s)$ be the maximum of $r_{\mathcal{P}}(F)$ over all $\mathcal{P}$-bad graphs $F$ with at most $s$ vertices; 
if no such graphs exist, we set $R_{\mathcal{P}}(s) = 0$ (this will not matter later on).
We are now ready to prove Theorem \ref{theo_rest3}, which we rephrase as follows.
\begin{proposition}\label{lem:large_number_of_vertices_tester}
	For every hereditary property $\mathcal{P}$ there are functions $M_{\mathcal{P}},s_{\mathcal{P}} : (0,1) \rightarrow \mathbb{N}$ such that for every $\varepsilon > 0$, the property $\mathcal{P}$ is $\varepsilon$-testable with one-sided error and sample complexity $s_{\mathcal{P}}(\varepsilon)$ under the promise that inputs $(G,\mathcal{D})$ always satisfy $|V(G)| \geq M_{\mathcal{P}}(\varepsilon)$.
\end{proposition}
\begin{proof}
	Consider the (extendable and hereditary) property $\mathcal{H} = \mathcal{H}(\mathcal{P})$ defined above.
	By Theorem \ref{statement:ours}, there is a function $s_{\mathcal{H}} : (0,1) \rightarrow \mathbb{N}$ such that for every $\varepsilon > 0$ and for every vertex-weighted graph $(G,\mathcal{D})$ which is $\varepsilon$-far from $\mathcal{H}$, a sample of $s$ vertices from $G$ (taken from $\mathcal{D}$) induces a subgraph which does not satisfy $\mathcal{H}$ with probability at least $\frac{2}{3}$.
	
	Our (``large inputs"-model) tester for $\mathcal{P}$ samples $s_{\mathcal{H}}(\varepsilon)$ vertices, and accepts if and only if the subgraph induced by the sample satisfies $\mathcal{H}$.
	We prove the proposition with
	$
	M = M_{\mathcal{P}}(\varepsilon) := R_{\mathcal{P}}(s_{\mathcal{H}}(\varepsilon)).
	$
	
	Let $(G,\mathcal{D})$ be a vertex-weighted graph with $|V(G)| \geq M$. Suppose first that $G$ satisfies $\mathcal{P}$. Our goal is to show that the subgraph induced by a sample of $s_{\mathcal{H}}(\varepsilon)$ vertices, taken from $\mathcal{D}$ and independently, satisfies $\mathcal{H}$ with probability $1$. So suppose by contradiction that $G$ contains an induced subgraph $F$ on at most $s_{\mathcal{H}}(\varepsilon)$ vertices which does not satisfy $\mathcal{H}$. In other words, $F$ is $\mathcal{P}$-bad. By the definition of $r_{\mathcal{P}}(F)$, there is no graph on $r_{\mathcal{P}}(F)$ vertices which satisfies $\mathcal{P}$ and contains $F$ as an induced subgraph. As $|V(G)| \geq M = R_{\mathcal{P}}(s_{\mathcal{H}}(\varepsilon)) \geq r_{\mathcal{P}}(F)$, and as $\mathcal{P}$ is hereditary, we get that $G$ does not satisfy $\mathcal{P}$, a contradiction.
	
	Suppose now that $(G,\mathcal{D})$ is $\varepsilon$-far from $\mathcal{P}$.
	Then $(G,\mathcal{D})$ is also $\varepsilon$-far from $\mathcal{H}$, as $\mathcal{H} \subseteq \mathcal{P}$. By our choice of
	$s_{\mathcal{H}}(\varepsilon)$, a sample of $s_{\mathcal{H}}(\varepsilon)$ vertices of $G$, taken from $\mathcal{D}$ and independently, does not satisfy $\mathcal{H}$ with probability at least $\frac{2}{3}$. So our tester rejects $(G,\mathcal{D})$ with probability at least $\frac{2}{3}$, as required.
\end{proof}
It is natural to ask whether we can replace the function $M_{\mathcal{P}}(\varepsilon)$ in Lemma \ref{lem:large_number_of_vertices_tester} by a constant depending only on $\mathcal{P}$ (and not on $\varepsilon$). As is shown in the following proposition, we cannot.

\begin{proposition}\label{prop:constant_bound_on_num_of_vertices}
	There is a hereditary property $\mathcal{P}$ such that for every $M > 0$, there is no tester for $\mathcal{P}$ in the VDF model even if we are guaranteed that the input graph has at least $M$ vertices.
\end{proposition}
\begin{proof}
	For each $k \geq 3$, let $C_k^*$ be the graph obtained from the $k$-cycle $C_k$ by adding an isolated vertex.
	Consider the property $\mathcal{P} = \{C_k^* : k \geq 3\}$-freeness. Let $M > 0$. Set $G = C_M$ and $G' = C_M^*$. Let $\mathcal{D}$ be the uniform distribution on $V(G)$, and let $\mathcal{D}'$ be the distribution on $V(G')$ which assigns weight $0$ to the isolated vertex in $G'$, and is uniform on the rest of the vertices of $G'$. Then $G \in \mathcal{P}$ and $(G',\mathcal{D}')$ is $\frac{1}{M^2}$-far from $\mathcal{P}$, but a sample (of any number of vertices) from $(G,\mathcal{D})$ is indistinguishable from a sample of the same size from $(G',\mathcal{D}')$. This shows that $\mathcal{P}$ is not testable even if we require input graphs to have at least $M$ vertices.
\end{proof}
\noindent
We now move on to prove Theorem \ref{theo_rest4}. 
\begin{proof}[Proof of Theorem \ref{theo_rest4}]
	Let $\mathcal{P}$ be a hereditary graph property. Our goal is to design (and prove the correctness of) a one-sided-error tester for $\mathcal{P}$ in the VDF model, provided that the tester receives $|V(G)|$ as part of the input. Let $M_{\mathcal{P}} : (0,1) \rightarrow \mathbb{N}$ be as in Lemma \ref{lem:large_number_of_vertices_tester}. On input $\varepsilon \in (0,1)$, $G$ and $\mathcal{D}$ (where $G$ is a graph and $\mathcal{D}$ is a distribution on $V(G)$), our tester works as follows:
	\begin{enumerate}
		\item If $|V(G)| \geq M_{\mathcal{P}}(\varepsilon)$, then invoke the tester whose existence is guaranteed by Lemma \ref{lem:large_number_of_vertices_tester}, and accept if and only if this tester accepts. 
		\item Otherwise, i.e. if $|V(G)| < M_{\mathcal{P}}(\varepsilon)$, then do the following: setting $M := M_{\mathcal{P}}(\varepsilon)$ and $t := M\log(3M)/\varepsilon$, sample vertices $u_1,\dots,u_t \in V(G)$ according to $\mathcal{D}$ and independently, and put $U := \{u_1,\dots,u_t\}$. Accept if and only if there exists a graph on $|V(G)|$ vertices which satisfies $\mathcal{P}$ and contains $G[U]$ as an induced subgraph (in the notation introduced at the beginning of this subsection, this is the same as saying that $r_{\mathcal{P}}(G[U]) > |V(G)|$). 
	\end{enumerate} 
	Let us prove the correctness of our tester. First, Lemma \ref{lem:large_number_of_vertices_tester} guarantees that if $|V(G)| \geq M_{\mathcal{P}}(\varepsilon)$ then the tester works correctly; namely, it accepts with probability $1$ if $G \in \mathcal{P}$, and rejects with probability at least $\frac{2}{3}$ if $(G,\mathcal{D})$ is $\varepsilon$-far from $\mathcal{P}$. 
	
	So from now on we may assume that $|V(G)| < M_{\mathcal{P}}(\varepsilon)$. 
	Suppose first that $G \in \mathcal{P}$. 
	Evidently, for every $U \subseteq V(G)$ there is a graph on $|V(G)|$ vertices which satisfies $\mathcal{P}$ and contains $G[U]$ as an induced subgraph (indeed, $G$ is such a graph). Hence, the tester accepts $G$ with probability $1$ (see Item 2). 
	
	Now suppose that $(G,\mathcal{D})$ is $\varepsilon$-far from $\mathcal{P}$. 
	Observe that for each $v \in V(G)$, the probability that $v \notin U$ is 
	$$
	\left( 1 - \mathcal{D}(v) \right)^t \leq e^{-\mathcal{D}(v) \cdot t} = 
	\left( \frac{1}{3M} \right)^{-\mathcal{D}(v) \cdot M/\varepsilon} \; .
	$$
	By taking the union bound over all (at most $|V(G)| < m$) vertices $v \in V(G)$ which satisfy 
	$\mathcal{D}(v) \geq \varepsilon/M$, we see that the probability that there is 
	$v \in V(G) \setminus U$ with $\mathcal{D}(v) \geq \varepsilon/M$, is at most $\frac{1}{3}$. Suppose that every $v \in V(G) \setminus U$ satisfies $\mathcal{D}(v) < \varepsilon/M$ (this happens with probability at least $\frac{2}{3}$). Then $\mathcal{D}(V(G) \setminus U) < |V(G)| \cdot \varepsilon/M < \varepsilon$ (where in the last inequality we used our assumption that $|V(G)| <  M$). Now, if (by contradiction) there is a graph $G'$ on $|V(G)|$ vertices which satisfies $\mathcal{P}$ and contains $G[U]$ as an induced subgraph, then one can turn $G$ into $G'$ by only adding/deleting edges which are incident to vertices in $V(G) \setminus U$. Since $\mathcal{D}(V(G) \setminus U) < \varepsilon$, this stands in contradiction to the assumption that $(G,\mathcal{D})$ is $\varepsilon$-far from $\mathcal{P}$. We conclude that there is no such graph $G'$. This implies that $(G,\mathcal{D})$ is rejected with probability at least $\frac{2}{3}$, as required. 
\end{proof}


Finally, we prove Theorem \ref{theo:rest2}, i.e. that every hereditary property is testable in the NLW model. We restate this theorem as follows.

\begin{proposition}\label{lem:no_small_weight_vertices_tester}
	For every hereditary property $\mathcal{P}$ there is a function
	$t_{\mathcal{P}} : (0,1)^2 \rightarrow \mathbb{N}$ such that for all $\varepsilon,\delta > 0$, the property
	$\mathcal{P}$ is $\varepsilon$-testable with one-sided error and sample complexity $t_{\mathcal{P}}(\varepsilon,\delta)$ under the promise that inputs $(G,\mathcal{D})$ always satisfy
	$\min_{v \in V(G)}{\mathcal{D}(v)} \geq \delta/|V(G)|$.
\end{proposition}
\begin{proof}
	We start by specifying the function $t_{\mathcal{P}}(\varepsilon,\delta)$. Consider the (extendable and hereditary) property $\mathcal{H} = \mathcal{H}(\mathcal{P})$ defined above. 
	By Theorem \ref{statement:ours}, there is a function $s_{\mathcal{H}} : (0,1) \rightarrow \mathbb{N}$ such that for every $\varepsilon > 0$ and for every vertex-weighted graph $(G,\mathcal{D})$ which is $\varepsilon$-far from $\mathcal{H}$, a sample of $s_{\mathcal{H}}(\varepsilon)$ vertices of $G$ (taken from $\mathcal{D}$) induces a subgraph which does not satisfy $\mathcal{H}$ with probability\footnote{The statement of Theorem \ref{statement:ours} only guarantees a success probability of $\frac{2}{3}$, but this can clearly be amplified to $\frac{5}{6}$ by repeating the experiment $O(1)$ times.} at least $\frac{5}{6}$. Now set
	$R := R_{\mathcal{P}}(s_{\mathcal{H}}(\varepsilon))$ and
	$$
	t = t_{\mathcal{P}}(\varepsilon,\delta) :=
	\max\left\{ s_{\mathcal{H}}(\varepsilon), \; 2R\log(6R)/\delta \right\}.
	$$
	
	Our tester for $\mathcal{P}$ in the NLW model simply samples a sequence of $t_{\mathcal{P}}(\varepsilon,\delta)$ vertices of the input and accepts if and only if the subgraph induced by the sample satisfies $\mathcal{P}$. Evidently, this tester accepts with probability $1$ if the input satisfies $\mathcal{P}$. So to establish the correctness of our tester, it suffices to show that it rejects with probability at least $\frac{2}{3}$ if the input $(G,\mathcal{D})$ is $\varepsilon$-far from $\mathcal{P}$.

	Let $\varepsilon,\delta > 0$, and let $(G,\mathcal{D})$ be a vertex-weighted graph on $n$ vertices which is $\varepsilon$-far from $\mathcal{P}$, and in which all vertices have weight at least $\delta/n$. Let $u_1,\dots,u_t$ be a sequence of $t = t_{\mathcal{P}}(\varepsilon,\delta)$ random vertices of $G$, sampled according to $\mathcal{D}$ and independently, and set $U = \{u_1,\dots,u_t\}$.  We need to show that with probability at least $\frac{2}{3}$, $G[U]$ does not satisfy $\mathcal{P}$.
	Suppose first that $n < 2R$. We claim that in this case we have $U = V(G)$ with probability at least $\frac{2}{3}$ (this is clearly sufficient because $G$ itself does not satisfy $\mathcal{P}$). For a vertex $v \in V(G)$, the probability that $u_i \neq v$ for every $1 \leq i \leq t$ is
	$$
	\left( 1 - \mathcal{D}(v) \right)^t \leq
	\left( 1 - \frac{\delta}{n} \right)^t <
	\left( 1 - \frac{\delta}{2R} \right)^t \leq
	e^{-\frac{\delta t}{2R}} \leq
	\frac{1}{6R} \; .
	$$
	So by the union bound over all $n < 2R$ vertices of $G$, we see that with probability at least $\frac{2}{3}$, $U = V(G)$, as required.
	
	Suppose now that
	$n \geq 2R$. 
	Our choice of $s = s_{\mathcal{H}}(\varepsilon)$ guarantees that with probability at least $\frac{5}{6}$, the graph $F := G[\{u_1,\dots,u_s\}]$ does not satisfy $\mathcal{H}$, meaning that it is $\mathcal{P}$-bad.  
	We will now show that with probability at least $\frac{5}{6}$, we have
	$|U| \geq R$. This will imply that with probability at least $\frac{2}{3}$, $G[U]$ contains as an induced subgraph a $\mathcal{P}$-bad graph $F$ on at most $s_{\mathcal{H}}(\varepsilon)$ vertices, and also
	$|U| \geq R =  R_{\mathcal{P}}(s_{\mathcal{H}}(\varepsilon)) \geq r_{\mathcal{P}}(F)$. By the definition of $r_{\mathcal{P}}(F)$, this would imply that $G[U]$ does not satisfy $\mathcal{P}$, as required.
	
	So from now on, our goal is to show that $|U| \geq R$ with probability at least $\frac{5}{6}$. Fix a partition of $V(G)$ into $R$ sets $V_1,\dots,V_R$, each of size at least
	$\lfloor \frac{n}{R} \rfloor \geq \frac{n}{2R}$. 
	For each $1 \leq i \leq R$, let $A_i$ be the event that $U \cap V_i \neq \emptyset$. 
	Note that if $A_i$ occurs for every $1 \leq i \leq R$, then $|U| \geq R$. Since 
	$\mathcal{D}(V_i) \geq |V_i| \cdot \frac{\delta}{n} \geq 
	\frac{n}{2R} \cdot \frac{\delta}{n} = \frac{\delta}{2R}$, the probability that $A_i$ does not occur is at most
	$$
	\left( 1 -  \mathcal{D}(V_i) \right)^{t} \leq
	\left( 1 - \frac{\delta}{2R} \right)^{t} \leq
	e^{-\frac{\delta t}{2R}} \leq \frac{1}{6R}.
	$$
	By the union bound, the probability that there is $1 \leq i \leq R$ for which $A_i$ does not occur, is at most $\frac{1}{6}$, as required. This completes the proof. 
\end{proof}
\subsection{Proof of Theorem \ref{theo:rest1}}\label{subsec:no_high_weight_vertices}
In this subsection we prove Theorem \ref{theo:rest1}, i.e. we show that every hereditary property is testable in the NHW model. 
Again, we rephrase as follows. 

\begin{proposition}\label{lem:no_high_weight_vertices_tester}
	For every hereditary property $\mathcal{P}$ there are functions $t_{\mathcal{P}} : (0,1) \rightarrow \mathbb{N}$ and
	$c_{\mathcal{P}} : \nolinebreak (0,1) \rightarrow (0,1)$ such that for every $\varepsilon > 0$, the property
	$\mathcal{P}$ is $\varepsilon$-testable with one-sided error and sample complexity $t_{\mathcal{P}}(\varepsilon)$ under the promise that inputs $(G,\mathcal{D})$ always satisfy
	$\max_{v \in V(G)}{\mathcal{D}(v)} \leq c_{\mathcal{P}}(\varepsilon)$.
\end{proposition}

The key idea in the proof of Proposition \ref{lem:no_high_weight_vertices_tester}, which appeared in \cite{Goldreich_VDF}, is to ``blow up" the vertex-weighted graph $(G,\mathcal{D})$ by replacing each vertex $v$ with a vertex-set whose size is proportional to $\mathcal{D}(v)$, and thus obtain an (unweighted) graph $G'$, to which one can apply known testability results in the standard model.

Let us introduce some definitions. For a graph $G$, say on $V(G) = \{v_1,\dots,v_n\}$, and for integers $b_1,\dots,b_n \geq 0$, a {\em $(b_1,\dots,b_n)$-blowup} of $G$ is any graph admitting a vertex-partition $V_1 \cup \dots \cup V_n$ such that
$|V_i| = b_i$ for every $1 \leq i \leq n$, and such that the bipartite graph between $V_i$ and $V_j$ is complete if $\{v_i,v_j\} \in E(G)$ and empty if $\{v_i,v_j\} \notin E(G)$. The sets $V_1,\dots,V_n$ are called the {\em blowup-sets}. Note that we do not pose any restrictions on the graphs induced by the sets $V_1,\dots,V_n$; these graphs may be arbitrary.
For simplicity of presentation, we assume henceforth that all vertex-weights 
are rational\footnote{\label{footnote:irrational_weights}If one allows general (i.e. possibly irrational) weights, then it is necessary to change the definition of a $(\mathcal{D},N)$-blowup by rounding $\mathcal{D}(v_i) \cdot N$ to the closest integer. This results in an additive error of $\frac{n}{N}$ in the conclusion of Lemma \ref{lem:reduction}, due to rounding. Consequently, in (the proofs of) Propositions \ref{lem:no_high_weight_vertices_tester} and \ref{thm:reduction} we need to consider $(\mathcal{D},N)$-blowups with $N \rightarrow \infty$ in order to have this error term go to $0$. We also need to replace $\varepsilon$ in several places with (say) $\frac{\varepsilon}{2}$ (or any other number smaller than $\varepsilon$). For example, the conclusion of Proposition \ref{thm:reduction} should be that $\mathcal{P}$ is testable in the VDF model by a tester having one-sided error and sample complexity $q_{\mathcal{P}}(\varepsilon/2)$.}. 
Now let $\mathcal{D}$ be a distribution on $V(G) = \{v_1,\dots,v_n\}$, and let $N \in \mathbb{N}$ be such that $\mathcal{D}(v_i) \cdot N$ is an integer for every $1 \leq i \leq n$; such an $N$ is called {\em admissible}. A {\em $(\mathcal{D},N)$-blowup} of $G$ is a $(b_1,\dots,b_n)$-blowup of $G$ with
$b_i = \mathcal{D}(v_i) \cdot N$ for every $1 \leq i \leq n$.
Note that a blowup is always treated as ``unweighted" (in other words, the distribution on its vertices is uniform). 
Goldreich \cite{Goldreich_VDF} proved that for every graph $F$ and $\varepsilon \in (0,1)$, if a vertex-weighted graph $(G,\mathcal{D})$ is $\varepsilon$-far from being $F$-free, then for every admissible $N$,
any $(\mathcal{D},N)$-blowup of $G$ is
$\frac{\varepsilon}{\binom{|V(F)|}{2}}$-far
from being $F$-free.
Goldreich further asked whether the $\binom{|V(F)|}{2}^{-1}$-factor can be avoided. In the following lemma we show that this is indeed the case, and moreover that an analogous statement holds for {\em every} hereditary property. This lemma is also the key ingredient in the proof of Proposition \ref{lem:no_high_weight_vertices_tester}.
\begin{lemma}\label{lem:reduction}
	Let $\mathcal{P}$ be a hereditary graph property and let $(G,\mathcal{D})$ be a vertex-weighted graph which is $\varepsilon$-far from $\mathcal{P}$. Then for every admissible $N$, any $(\mathcal{D},N)$-blowup of $G$ is $\varepsilon$-far from $\mathcal{P}$.
\end{lemma}
\begin{proof}
	Fix any admissible $N$ and let $G'$ be a $(\mathcal{D},N)$-blowup of $G$. As above, we use $v_1,\dots,v_n$ to denote the vertices of $G$, and $V_1,\dots,V_n$ to denote the corresponding blowup sets.
	Suppose by contradiction that there is a graph $H'$ on $V(G')$ that satisfies $\mathcal{P}$ and is $\varepsilon$-close to $G'$. Let $H$ be the random graph defined as follows: the vertex-set of $H$ is $V(H) = V(G) = \{v_1,\dots,v_n\}$. To define the edge-set of $H$, sample for each $1 \leq i \leq n$ a vertex $u_i \in V_i$ uniformly at random, and make $\{v_i,v_j\}$ an edge in $H$ if and only if $\{u_i,u_j\}$ is an edge in $H'$ (for $1 \leq i < j \leq n$). Then $H$ satisfies $\mathcal{P}$ (with probability $1$) because $H$ is isomorphic to an induced subgraph of $H'$ and $\mathcal{P}$ is hereditary. Let us compute the expected distance between $H$ and $G$ (here the distance is with respect to the distribution $\mathcal{D}$). For each $1 \leq i < j \leq n$, the probability that $\{v_i,v_j\} \in E(G) \triangle E(H)$ \nolinebreak is \nolinebreak precisely
	$$
	\frac{|E_{G'}(V_i,V_j) \triangle E_{H'}(V_i,V_j)|}{|V_i||V_j|} = \frac{|E_{G'}(V_i,V_j) \triangle E_{H'}(V_i,V_j)|}
	{\mathcal{D}(v_i)\mathcal{D}(v_j)N^2} \; .
	$$
	Hence, the expected distance between $H$ and $G$ is
	\begin{align*}
	&\sum_{1 \leq i < j \leq n}{\mathcal{D}(v_i)\mathcal{D}(v_j) \cdot \frac{|E_{G'}(V_i,V_j) \triangle E_{H'}(V_i,V_j)|}
		{\mathcal{D}(v_i)\mathcal{D}(v_j)N^2}}
	= \\&\frac{1}{N^2}\sum_{1 \leq i < j \leq n}{|E_{G'}(V_i,V_j) \triangle E_{H'}(V_i,V_j)|} < \varepsilon,
	\end{align*}
	where the last inequality uses the assumption that $G'$ is $\varepsilon$-close to $H'$. So $G$ is $\varepsilon$-close to a graph $H$ which satisfies $\mathcal{P}$, a contradiction.
\end{proof}

By combining Lemma \ref{lem:reduction} with the result of \cite{AS_hereditary} (that all hereditary properties are testable with one-sided error in the standard model), we obtain the following: for every hereditary property $\mathcal{P}$, for every vertex-weighted graph $(G,\mathcal{D})$ which is $\varepsilon$-far from $\mathcal{P}$, for every admissible $N$ and for every $(\mathcal{D},N)$-blowup $G'$ of $G$, it holds that $G'$ is $\varepsilon$-far from $\mathcal{P}$ with respect to the uniform distribution, and hence a sample of some $s = s_{\mathcal{P}}(\varepsilon)$ vertices of $G'$, taken uniformly and independently, induces a graph which w.h.p. does not satisfy $\mathcal{P}$. Observe that this induced subgraph of $G'$ has (essentially) the same distribution as the graph $S$ on $[s]$ obtained by sampling vertices $u_1,\dots,u_s \in V(G)$ from $\mathcal{D}$ independently, and letting $\{i,j\} \in E(S)$ if and only if $\{u_i,u_j\} \in E(S)$ (this is precisely the graph defined in Theorem \ref{statement:LS}). We thus established Theorem \ref{statement:LS}, as promised \nolinebreak in \nolinebreak Subsection \nolinebreak \ref{subsec:LS}.

As noted in Subsection \ref{subsec:LS}, the graph $S$ defined above is a {\em blowup} of an induced subgraph of $G$, but is not necessarily a subgraph of $G$ in itself. This is because the sequence $u_1,\dots,u_s$ might contain repeated vertices. In other words, it may be the case that $G'$ contains ``forbidden subgraphs" which use several vertices from one of the blowup-sets, and thus do not correspond to ``forbidden subgraphs" in $G$. This creates an obstacle for proving Proposition \ref{lem:no_high_weight_vertices_tester}, because in order to prove this proposition we need to know that a (suitably chosen) random induced subgraph of $G$ (and not just the blowup thereof) does not satisfy $\mathcal{P}$ w.h.p. To avoid this obstacle, we use the assumption that all vertices in $G$ have relatively small weight, which guarantees that it is unlikely to sample more than once from some blowup-set (or in other words, that $S$ is isomorphic to $G[\{u_1,\dots,u_s\}]$.). We note that a different way of dealing with this obstacle is to restrict ourselves to properties for which we can guarantee, by appropriately choosing the graphs inside the blowup-sets, that there would not be any minimal forbidden subgraph which uses several vertices from one of the blowup-sets, see Subsection \nolinebreak \ref{subsec:reduction}.

\begin{proof}[Proof of Proposition \ref{lem:no_high_weight_vertices_tester}]
	We start by specifying the functions $t_{\mathcal{P}}$ and
	$c_{\mathcal{P}}$. 
	By the main result of \cite{AS_hereditary}, there is a function
	$q_{\mathcal{P}} : (0,1) \rightarrow \mathbb{N}$ such that for every $\varepsilon > 0$ and for every (unweighted) graph $G$ which is $\varepsilon$-far from $\mathcal{P}$, a sample of $q_{\mathcal{P}}(\varepsilon)$ vertices from $G$, taken {\em uniformly at random} and independently, induces a graph which does not satisfy $\mathcal{P}$ with probability at least $\frac{5}{6}$. Now set $t_{\mathcal{P}}(\varepsilon) := q_{\mathcal{P}}(\varepsilon)$ and
	$$
	c_{\mathcal{P}}(\varepsilon) := \frac{1}{3q_{\mathcal{P}}^2(\varepsilon)}.
	$$
	
	Our tester for $\mathcal{P}$ in the NHW model simply samples a sequence of $t = t_{\mathcal{P}}(\varepsilon)$ vertices of the input and accepts if and only if the subgraph induced by the sample satisfies $\mathcal{P}$. Evidently, this tester accepts with probability $1$ if the input satisfies $\mathcal{P}$. So to establish the correctness of our tester, it suffices to show that it rejects with probability at least $\frac{2}{3}$ if the input $(G,\mathcal{D})$ is $\varepsilon$-far from $\mathcal{P}$. 
	
	Let $\varepsilon > 0$ and let $(G,\mathcal{D})$ be a vertex-weighted graph on $n$ vertices which is $\varepsilon$-far from $\mathcal{P}$, and in which all vertices have weight at most $c$, where
	$c = c_{\mathcal{P}}(\varepsilon)$. Write $V(G) = \{v_1,\dots,v_n\}$ and fix an admissible $N$, that is, a positive integer $N$ such that $\mathcal{D}(v_i) \cdot N$ is an integer for every $1 \leq i \leq n$.
	Let $G'$ be an arbitrary $(\mathcal{D},N)$-blowup of $G$, and denote the blowup-sets by $V_1,\dots,V_n$.
	By Lemma \ref{lem:reduction}, $G'$ is $\varepsilon$-far from $\mathcal{P}$. This implies that a random sequence $u_1,\dots,u_q$ of $q = q_{\mathcal{P}}(\varepsilon)$ vertices of $G'$, sampled uniformly and independently, induces a graph which does not satisfy $\mathcal{P}$ with probability at least $\frac{5}{6}$.
	
	Let $\varphi : V(G') \rightarrow V(G)$ be the map which maps all elements of $V_i$ to $v_i$ (for every $1 \leq i \leq n$). Observe that for $u \in V(G')$ sampled uniformly, the random vertex $\varphi(u) \in V(G)$ has the distribution $\mathcal{D}$ (because $|V_i| = \mathcal{D}(v_i) \cdot N = \mathcal{D}(v_i) \cdot |V(G')|$).
	Furthermore, if a set $U \subseteq V(G')$ satisfies $|V_i \cap U| \leq 1$ for every $1 \leq i \leq n$, then $G[\varphi(U)]$ is isomorphic to $G'[U]$. Let $u_1,\dots,u_q$ be a random sequence of vertices of $G'$, sampled uniformly and independently, and set $U := \{u_1,\dots,u_q\}$. Recall that $G'[U]$ does not satisfy $\mathcal{P}$ with probability at least $\frac{5}{6}$. Furthermore, the probability that
	$|V_i \cap U| \geq 2$ for some $1 \leq i \leq n$ is at most
	$$
	\sum_{i = 1}^{n}{\binom{q}{2} \cdot \mathcal{D}^2(v_i)} \leq
	\frac{q^2}{2} \cdot c \cdot \sum_{i = 1}^{n}{\mathcal{D}(v_i)} =
	\frac{q^2}{2} \cdot c  = \frac{1}{6}.
	$$
	We conclude that with probability at least $\frac{2}{3}$, $G'[U]$ does not satisfy $\mathcal{P}$ and $|V_i \cap U| \leq 1$ for every $1 \leq i \leq n$, implying that $G[\varphi(U)]$ does not satisfy $\mathcal{P}$ either. This completes the proof.
\end{proof}

It is natural to ask whether the function $c_{\mathcal{P}}(\varepsilon)$ from Proposition \ref{lem:no_high_weight_vertices_tester} needs to depend on $\varepsilon$, namely whether the statement of Proposition \ref{lem:no_high_weight_vertices_tester} holds even if $c_{\mathcal{P}}$ is a constant function (depending only on $\mathcal{P}$). The proof of Proposition \ref{prop:constant_bound_on_num_of_vertices} shows, however, that this is not the case. In other words, allowing $c_{\mathcal{P}}(\varepsilon)$ to depend on $\varepsilon$ is unavoidable.
\subsection{Testing in the VDF Model vs. Testing in the Standard Model}\label{subsec:reduction}

It is natural to ask about the relation between the sample complexity for testing a property in the VDF model and the sample complexity for testing it in the standard model. More specifically, it will be interesting to resolve the following:

\begin{problem}\label{prob:reduction}
	Is it true that every extendable hereditary property $\mathcal{P}$ can be tested in the VDF model with the same (or close to the same) sample complexity as in the (standard) dense graph model?
\end{problem}

While at present we cannot answer this question, we can show that many natural properties $\mathcal{P}$ can be tested in the VDF model with (exactly) the same sample complexity as that of the (optimal) tester for $\mathcal{P}$ in the standard model, which works by sampling a certain number of vertices and accepting if and only if they induce a graph which satisfies $\mathcal{P}$. This is explained in the following paragraph.

As mentioned in Subsection \ref{subsec:no_high_weight_vertices}, the assumption made in Proposition \ref{lem:no_high_weight_vertices_tester} regarding the non-existence of high-weight vertices is needed in order to handle the possibility of having copies of some (forbidden) graph $F$ in $G'$ which do not correspond to copies of $F$ in $G$ (where $G'$ is some blowup of $G$). For some graph properties, however, such an assumption is not required, as we can make sure that every copy of a minimal forbidden graph in $G'$ will correspond to such a copy in $G$. To make this precise, we need the following definition.
A family of graphs $\mathcal{F}$ is said to be {\em blowup-avoidable} if for every graph $G$, say on $\{v_1,\dots,v_n\}$, and for every $n$-tuple of integers $b_1,\dots,b_n \geq 0$, there is a $(b_1,\dots,b_n)$-blowup $G'$ of $G$ with blowup-sets $V_1,\dots,V_n$, such that there is no induced copy of any $F \in \mathcal{F}$ in $G'$ which intersects some $V_i$ in at least $2$ vertices; in other words, for every $F \in \mathcal{F}$, every induced copy of $F$ in $G'$ corresponds to an induced copy of $F$ in $G$. We say that a hereditary property $\mathcal{P}$ is {\em blowup-avoidable} if the family of minimal forbidden induced subgraphs for $\mathcal{P}$ is blowup-avoidable.
We now prove the following proposition, which partially resolves Problem \ref{prob:reduction}. The proof is similar to that of Proposition \ref{lem:no_high_weight_vertices_tester}.
\begin{proposition}\label{thm:reduction}
	Let $\mathcal{P}$ be a hereditary graph property which is blowup-avoidable, and suppose that $\mathcal{P}$ admits a tester in the standard model, which works by sampling $q_{\mathcal{P}}(\varepsilon)$ vertices uniformly at random and independently, and accepting if and only if the subgraph induced by the sample satisfies $\mathcal{P}$. Then $\mathcal{P}$ is testable in the VDF model by a tester having one-sided error and sample complexity\footnote{Provided that the input distributions are only allowed to assign rational weights. If irrational weights are allowed, then the sample complexity (of the VDF tester for $\mathcal{P}$) should be slightly increased to (say) $q_{\mathcal{P}}(\varepsilon/2)$, see Footnote \ref{footnote:irrational_weights}.} \nolinebreak $q_{\mathcal{P}}(\varepsilon)$.
\end{proposition}
\begin{proof}
	Given an input $(G,\mathcal{D})$,
	the required VDF tester for $\mathcal{P}$ samples (from $\mathcal{D}$) a sequence of $q_{\mathcal{P}}(\varepsilon)$ vertices, and accepts if and only if the subgraph induced by the sample satisfies $\mathcal{P}$. Since $\mathcal{P}$ is hereditary, this tester accepts with probability $1$ if the input graph satisfies $\mathcal{P}$. So it remains to show that if the input $(G,\mathcal{D})$ is $\varepsilon$-far from $\mathcal{P}$, then with probability at least $\frac{2}{3}$, a sequence of $q_{\mathcal{P}}(\varepsilon)$ vertices of $G$, sampled according to $\mathcal{D}$ and independently, induces a graph which does not satisfy $\mathcal{P}$.
	
	Let $\mathcal{F} = \mathcal{F}(\mathcal{P})$ be the family of minimal forbidden induced subgraphs for $\mathcal{P}$.
	Let $(G,\mathcal{D})$ be a vertex-weighted graph on $n$ vertices, which is $\varepsilon$-far from $\mathcal{P}$. Write $V(G) = \{v_1,\dots,v_n\}$ and fix an admissible $N$, that is, a positive integer $N$ such that $\mathcal{D}(v_i) \cdot N$ is an integer for every $1 \leq i \leq n$. As $\mathcal{P}$ is blowup-avoidable, there is a
	$(\mathcal{D},N)$-blowup $G'$ of $G$ with blowup-sets $V_1,\dots,V_n$, such that there is no induced copy of any $F \in \mathcal{F}$ in $G'$ which intersects some $V_i$ in at least $2$ vertices. By Lemma \ref{lem:reduction}, $G'$ is $\varepsilon$-far from $\mathcal{P}$. So by our choice of $q_{\mathcal{P}}(\varepsilon)$, with probability at least $\frac{2}{3}$ it holds that a sequence of $q_{\mathcal{P}}(\varepsilon)$ vertices of $G'$, sampled uniformly and independently, induces a graph which does not satisfy $\mathcal{P}$, and hence contains an induced copy of some $F \in \mathcal{F}$.
	
	Let $\varphi : V(G') \rightarrow V(G)$ be the map which maps all elements of $V_i$ to $v_i$ (for every $1 \leq i \leq n$). Observe that for $u \in V(G')$ sampled uniformly, the random vertex $\varphi(u) \in V(G)$ has the distribution $\mathcal{D}$. 
	Note that by our choice of $G'$, if $u_1,\dots,u_r \in V(G')$ span an induced copy of some $F \in \mathcal{F}$ (in the graph $G'$), then $\varphi|_{\{u_1,\dots,u_r\}}$ is injective (and hence an isomorphism), which implies that $\varphi(u_1),\dots,\varphi(u_r)$ span an induced copy of $F$ in $G$.  
	It is now easy to see that a sequence of $q_{\mathcal{P}}(\varepsilon)$ vertices of $G$, sampled from $\mathcal{D}$ and independently, does not satisfy $\mathcal{P}$ with probability at least $\frac{2}{3}$, as required.
\end{proof}

To demonstrate the usefulness of Proposition \ref{thm:reduction}, observe that induced $F$-freeness is blowup-avoidable for every $F \in \{P_2,P_3,C_4\}$ (here $P_k$ is the path with $k$ edges). Indeed, this is established by taking the blowup-sets (in the definition of blowup-avoidability) to be cliques. By combining Proposition \ref{thm:reduction} with known results for the standard model \cite{AS_hereditary,AF,GS_C4}, we immediately get that induced $F$-freeness is testable in the VDF model with sample complexity $\poly(1/\varepsilon)$ if $F \in \{P_2,P_3\}$, and with sample complexity at most $2^{\poly(1/\varepsilon)}$ if \nolinebreak $F = C_4$.

We now describe another corollary of Proposition \ref{thm:reduction}.
We say that a graph property $\mathcal{P}$ is {\em closed under blowups} if for every graph $G$ satisfying $\mathcal{P}$, every blowup of $G$ in which the blowup-sets are independent sets also satisfies $\mathcal{P}$.
We claim that if a hereditary property $\mathcal{P}$ is closed under blowups then it is also blowup-avoidable. To see this, let $\mathcal{F}$ be the set of minimal forbidden induced subgraphs for $\mathcal{P}$, let $G$ be an $n$-vertex graph, let $b_1,\dots,b_n \geq 0$ be integers and let $G'$ be the $(b_1,\dots,b_n)$-blowup of $G$ in which the blowup-sets, $V_1,\dots,V_n$, are independent.
Let $F \in \mathcal{F}$ and suppose that $G'$ contains an induced copy of $F$.
If, by contradiction, this copy intersects some $V_i$ in more than one vertex, then $F$ is a blowup of some graph $F'$ with $|V(F')| < |V(F)|$, where the blowup-sets are independent sets. Since $\mathcal{P}$ is closed under blowups and $F \notin \mathcal{P}$, we must have $F' \notin \mathcal{P}$; but this contradicts the fact that $F$ is a minimal forbidden induced subgraph for $\mathcal{P}$.

So we see that the conclusion of Proposition \ref{thm:reduction} applies to hereditary properties which are closed under blowups.
Some examples of such properties include $K_t$-freeness;
the property of having a homomorphism into a fixed graph $H$ (and in particular the property of being $k$-colorable); and the property of being the blowup of a fixed graph $H$ (cf. \cite{AG}).

On the negative side, there are many natural hereditary properties which are extendable but not blowup-avoidable, such as the property of being $H$-free for a graph $H$ which is neither a clique nor contains isolated vertices. It would be interesting to resolve Problem \ref{prob:reduction} for these properties.

\subsection{Which Properties are Testable in the Variations of the VDF Model?}\label{subsec:restriced_models_open_problems}

It may be interesting to characterize the graph properties which are testable in each of the variations of the VDF model (defined at the beginning of Section \ref{sec:restricted_models}).
\begin{problem}\label{prob:no_low_weight_vertices}
	Which graph properties are testable in the ``large inputs"/``size-aware"/NHW/NLW model?
\end{problem}
While at the moment we are unable to resolve Problem \ref{prob:no_low_weight_vertices}, we can rule out some initial guesses. A first guess might be that only hereditary properties are testable in these models. This, however, turns out to be false; for example, connectivity and hamiltonicity are testable
in each of these models, as implied by the following \nolinebreak proposition.
\begin{proposition}\label{prop:trivial_properties}
	Let $\mathcal{P}$ be a property such that for every $\varepsilon > 0$ there is $M(\varepsilon)$ so that every vertex-weighted graph on at least $M(\varepsilon)$ vertices is $\varepsilon$-close to $\mathcal{P}$. Then $\mathcal{P}$ is testable in all four variations of the VDF \nolinebreak model.
\end{proposition}
\begin{proof}
	The fact that $\mathcal{P}$ is testable in the ``large inputs" (resp. NHW) model is trivial; indeed, by choosing $M_{\mathcal{P}}(\varepsilon) := M(\varepsilon)$ (resp. $c_{\mathcal{P}}(\varepsilon) := 1/M(\varepsilon)$) we can make sure that every input graph will be $\varepsilon$-close to $\mathcal{P}$, so a tester that simply accepts without making any queries is a valid tester for $\mathcal{P}$. 
	
	Let us now consider the NLW model.
	Given $\varepsilon,\delta > 0$ and an input graph $(G,\mathcal{D})$ with all vertex-weights at least $\frac{\delta}{|V(G)|}$, our tester for $\mathcal{P}$ works as follows: setting $M := M(\varepsilon)$, the tester samples
	$C M\log(M)/\delta $ vertices according to $\mathcal{D}$ and independently (where $C$ is some large constant); if the number of distinct vertices in the sample is at least $M$ then the tester accepts (without making any queries), and otherwise the tester accepts if and only if the subgraph induced by the sample satisfies $\mathcal{P}$. To see that this is a valid tester, observe that if $G$ has less than $M$ vertices then \whp\, the tester samples all the vertices, and if $G$ has at least $M$ vertices then \whp\, there are at least $M$ distinct vertices in the sample. This can be argued similarly as in the proof of Proposition \ref{lem:no_small_weight_vertices_tester}, using that all vertices have weight at least $\frac{\delta}{|V(G)|}$; we omit the details. 
	
	Finally, let us prove that $\mathcal{P}$ is testable in the ``size-aware" model. On input $\varepsilon > 0$ and $(G,\mathcal{D})$, our tester for $\mathcal{P}$ (in the ``size-aware" model) does the following: if $|V(G)| \geq M(\varepsilon)$ then the tester accepts without making any queries, and if $|V(G)| < M(\varepsilon)$ then the tester samples $t := M\log(3M)/\varepsilon$ vertices $u_1,\dots,u_t \in V(G)$ according to the distribution $\mathcal{D}$ and independently, where $M = M(\varepsilon)$, and accepts if and only if there is a graph on $|V(G)|$ vertices which satisfies $\mathcal{P}$ and contains $G[\{u_1,\dots,u_t\}]$ as an induced subgraph. The proof of correctness for this tester is similar to the proof of Theorem \ref{theo_rest4}, and we leave the details to the reader. 
\end{proof}
In order to apply Proposition \ref{prop:trivial_properties} to the properties of connectivity and hamiltonicity, we observe that any vertex-weighted graph $(G,\mathcal{D})$ with
$|V(G)| \geq 1/\varepsilon$ is $\varepsilon$-close to being hamiltonian (and hence also connected). 
To see that this holds, take a random (cyclic) ordering $v_1,\dots,v_n$ of the vertices of $G$, and observe that for every pair of distinct $u,w \in V(G)$, the probability that there is $1 \leq i \leq n$ such that $\{u,w\} = \{v_i,v_{i+1}\}$ is $n/\binom{n}{2} = \frac{2}{n-1}$. This implies that the expected value of $\sum_{i=1}^{n}{\mathcal{D}(v_i)\mathcal{D}(v_{i+1})}$ is 
$\frac{2}{n-1} \cdot \sum_{u,w \in V(G)}{\mathcal{D}(u)\mathcal{D}(w)} = 
\frac{2}{n-1} \cdot \frac{1}{2} \cdot \left( 1 - \sum_{v \in V(G)}{\mathcal{D}(v)^2} \right) \leq 
\frac{1}{n-1} \cdot \left( 1 - \frac{1}{n} \right) = \frac{1}{n}$,
where the last inequality follows from Cauchy-Schwarz
(and the first sum is over {\em unordered} pairs $\{u,w\}$). 
This means that we can create a hamilton cycle by adding edges of total weight at most $\frac{1}{n} \leq \varepsilon$.  
Let us also note that for connectivity there is a simpler argument: if $(G,\mathcal{D})$ is a vertex-weighted graph with $|V(G)| \geq 1/\varepsilon$, then there is $v \in V(G)$ with $\mathcal{D}(v) \leq \varepsilon$, and we can make $G$ connected by connecting $v$ to all other vertices.

Note that in some of the restricted models (e.g. the NLW model), the tester given by (the proof of) Proposition \ref{prop:trivial_properties} has 2-sided error. It is also not hard to see that the NLW model admits no 1-sided-error tester for, e.g., connectivity. This shows that (some of) the restricted models allow for properties which are testable with 2-sided error but not with 1-sided error (unlike the ``ordinary" VDF model, where we know that every testable property can be tested with $1$-sided error, as follows from Theorems \ref{thm:main} and \ref{statement:ours}; see also \nolinebreak \cite[\nolinebreak Theorem \nolinebreak 2.3]{Goldreich_VDF}).

Another natural guess regarding the answer to Problem \ref{prob:no_low_weight_vertices} would be that every property which is testable in the standard model is also testable in the restricted models (see \cite{AFNS} for a characterization of the properties testable in the standard model). This guess is ruled out by the following proposition, which describes a property which is testable in the standard model but not in the restricted models.
\begin{proposition}\label{prop:impossibility_restricted_VDF}
	The property $\mathcal{P}$ of having edge-density\footnote{The edge-density of a (possibly vertex-weighted) graph $G$ is defined as $2e(G)/|V(G)|^2$; in other words, the density is defined with respect to the uniform distribution on $V(G)$, and {\em not} with respect to the given distribution $\mathcal{D}$.} at most $\frac{1}{4}$ is not testable in either of the four variants of the VDF model.
\end{proposition}
\begin{proof}
	Let $G_1$ be the $n$-vertex graph consisting of a clique of size $\frac{n}{2}$ and $\frac{n}{2}$ isolated vertices, and let $\mathcal{D}_1$ be the uniform distribution on $V(G_1)$. Let $G_2$ be the $n$-vertex graph consisting of a clique $X$ of size $\frac{3n}{4}$ and $\frac{n}{4}$ isolated vertices, and let $\mathcal{D}_2$ be the distribution on $V(G_2)$ that assigns weight $\frac{2}{3n}$ to every vertex of $X$, and weight $\frac{2}{n}$ to every vertex of $V(G_2) \setminus X$.
	Note that $(G_1,\mathcal{D}_1)$ and $(G_2,\mathcal{D}_2)$ are valid inputs in each of the variants of the VDF model (provided that $n$ is large enough), and that $G_1$ satisfies $\mathcal{P}$ while $(G_2,\mathcal{D}_2)$ is $\Omega(1)$-far from $\mathcal{P}$.
	On the other hand, we now show that for every $q$, a sample of $q$ vertices from $(G_1,\mathcal{D}_1)$ is indistinguishable from a sample of $q$ vertices from $(G_2,\mathcal{D}_2)$ (provided that $n$ is large enough with respect to $q$). To this end, let $U_i$ be a set of $q$ random vertices of $G_i$ sampled according to $\mathcal{D}_i$ and independently (for $i = 1,2$).
	Then for both $i=1,2$, the graph $G_i[U_i]$ consists of a clique and some isolated vertices.
	Letting $X_i$ be the clique in $G_i[U_i]$, we have
	\begin{align*}
	\mathbb{P}[|X_1| = k] &=
	o(1) +
	\binom{q}{k} \cdot \prod_{i=0}^{k-1}{\left( \frac{n}{2} - i \right)} \cdot
	\prod_{i=0}^{q-k-1}{\left( \frac{n}{2} - i \right)} \cdot
	\left( \frac{1}{n} \right)^{q}
	\\ &=
	(1 + o(1))\binom{q}{k} \cdot \left( \frac{1}{2} \right)^q
	\end{align*}
	and
	\begin{align*}
	\mathbb{P}[|X_2| = k] &=
	o(1) +
	\binom{q}{k} \cdot \prod_{i=0}^{k-1}{\left( \frac{3n}{4} - i \right)} \cdot
	\prod_{i=0}^{q-k-1} {\left( \frac{n}{4} - i \right)} \cdot
	\left( \frac{2}{3n} \right)^{k} \cdot \left( \frac{2}{n} \right)^{q-k}
	\\ &=
	(1 + o(1))\binom{q}{k} \cdot \left( \frac{1}{2} \right)^q,
	\end{align*}
	where in both cases, the additive term $o(1)$ accounts for the event that some vertex has been sampled more than once.
	So we see that
	$\left| \mathbb{P}[|X_1| = k] - \mathbb{P}[|X_2| = k] \right| = o(1)$. This implies that the total variation distance between the distribution of $G_1[U_1]$ and the distribution of $G_2[U_2]$ is $o(1)$. It follows that $\mathcal{P}$ is not testable in any of the four variants of the VDF model (note that knowing $n$ does not help to distinguish between $(G_1,\mathcal{D}_1)$ and $(G_2,\mathcal{D}_2)$, since these graphs have the same number of vertices). 
\end{proof}

The proof of Proposition \ref{prop:impossibility_restricted_VDF} can be adapted to show that other properties are also not testable in either of the variants of the VDF model. These properties include the property of having a cut with at least $\alpha n^2$ edges (for $0 < \alpha < \frac{1}{4}$) and the property of containing a clique with at least $\alpha n$ vertices (for $0 < \alpha < 1$).

\paragraph{Acknowledgements} We are grateful to an anonymous referee for spotting a gap in the proof of Theorem \ref{thm:main} in a preliminary version of the paper.

\section{Proof of Lemmas \ref{lem:reg} and \ref{lem:strong_reg}}\label{sec:appendix}
Here we prove lemmas \ref{lem:reg} and \ref{lem:strong_reg}. We start by extending some basic results about regular partitions to the vertex-weighted setting.
\begin{lemma}\label{lem:density_mean_variance}
	Let $X,Y$ be disjoint vertex-sets in a vertex-weighted graph $(G,\mathcal{D})$, and let $\mathcal{P}_X,\mathcal{P}_Y$ be partitions of $X,Y$, respectively. Then
	$$
	\sum_{X' \in \mathcal{P}_X, Y' \in \mathcal{P}_Y}
	{\mathcal{D}(X')\mathcal{D}(Y') \cdot d(X',Y')}
	= \mathcal{D}(X)\mathcal{D}(Y) \cdot d(X,Y),
	$$
	and
	\begin{align*}
	\sum_{X' \in \mathcal{P}_X, Y' \in \mathcal{P}_Y}&
	{\mathcal{D}(X')\mathcal{D}(Y') \cdot d^2(X',Y')}
	= \\&
	\mathcal{D}(X)\mathcal{D}(Y) \cdot d^2(X,Y) +
	\sum_{X' \in \mathcal{P}_X, Y' \in \mathcal{P}_Y}{\mathcal{D}(X')\mathcal{D}(Y') \cdot \left(d(X',Y') - d(X,Y)\right)^2}.
	\end{align*}
\end{lemma}
\begin{proof}
	We start with the first part of the lemma.
	\begin{align*}
	\sum_{X' \in \mathcal{P}_X, Y' \in \mathcal{P}_Y}
	{\mathcal{D}(X')\mathcal{D}(Y') \cdot d(X',Y')} &=
	\sum_{X' \in \mathcal{P}_X, Y' \in \mathcal{P}_Y}
	{\sum_{(x,y) \in E(X',Y')}{\mathcal{D}(x)\mathcal{D}(y)}} \\ &=
	\sum_{(x,y) \in E(X,Y)}{\mathcal{D}(x)\mathcal{D}(y)} = \mathcal{D}(X)\mathcal{D}(Y) \cdot d(X,Y).
	\end{align*}
	To prove the second part, we set $\varepsilon(X',Y') = d(X',Y') - d(X,Y)$ for each $X' \in \mathcal{P}_X$, $Y' \in \mathcal{P}_Y$. Now, 
	\begin{align*}
	&\sum_{X' \in \mathcal{P}_X, Y' \in \mathcal{P}_Y}
	{\mathcal{D}(X')\mathcal{D}(Y') \cdot d^2(X',Y')} =
	\sum_{X' \in \mathcal{P}_X, Y' \in \mathcal{P}_Y}
	{\mathcal{D}(X')\mathcal{D}(Y') \cdot \left(d(X,Y) + \varepsilon(X',Y') \right)^2}
	= \\&
	\sum_{X' \in \mathcal{P}_X, Y' \in \mathcal{P}_Y}
	{\mathcal{D}(X')\mathcal{D}(Y') \cdot
		\left( d^2(X,Y) + 2d(X,Y)\varepsilon(X',Y') + \varepsilon^2(X',Y') \right)}
	= \\&
	\mathcal{D}(X)\mathcal{D}(Y)d^2(X,Y) + 	
	\sum_{X' \in \mathcal{P}_X, Y' \in \mathcal{P}_Y}
	{\mathcal{D}(X')\mathcal{D}(Y') \cdot \varepsilon^2(X',Y')},
	\end{align*}
	where in the last equality we used the first part of the lemma.
\end{proof}	
Let $(G,\mathcal{D})$ be a vertex-weighted graph, and let $\mathcal{P} = \{P_1,\dots,P_r\}$ be a partition of $V(G)$. The index of $\mathcal{P}$, denoted $q(\mathcal{P})$, is defined as
$$
q(\mathcal{P}) = \sum_{1 \leq i < j \leq r}{\D(P_i)\D(P_j) \cdot d^2(P_i,P_j)}.
$$
\begin{lemma}\label{lem:index_refinement}
	For every vertex-partition $\mathcal{P}$ of a vertex-weighted graph $(G,\mathcal{D})$, and for every refinement $\mathcal{P}'$ of $\mathcal{P}$, we have $q(\mathcal{P}') \geq q(\mathcal{P})$.
\end{lemma}
\begin{proof}
	Write $\mathcal{P} = \{P_1,\dots,P_r\}$, and for each $1 \leq i \leq r$ put
	$\mathcal{P}'_i = \{P' \in \mathcal{P}' : P' \subseteq P_i\}$. Then
	$$
	q(\mathcal{P}') \geq \sum_{1 \leq i < j \leq r}{
		\sum_{P'_i \in \mathcal{P}'_i,P'_j \in \mathcal{P}'_j}
		{\mathcal{D}(P'_i)\mathcal{D}(P'_j) \cdot d^2(P'_i,P'_j)}
	} \geq
	\sum_{1 \leq i < j \leq r}
	{
		\mathcal{D}(P_i)\mathcal{D}(P_j) \cdot d^2(P_i,P_j)
	} =
	q(\mathcal{P}),
	$$
	where in the second inequality we used the second part of Lemma \ref{lem:density_mean_variance}.
\end{proof}
\begin{lemma}\label{lem:index_increase}
	Let $(G,\mathcal{D})$ be a vertex-weighted graph and let $\mathcal{P} = \{P_1,\dots,P_r\}$ be a non-$\varepsilon$-regular partition of $V(G)$. Then there is a refinement $\mathcal{P}'$ of $\mathcal{P}$ such that $|\mathcal{P}'| \leq |\mathcal{P}|\cdot 2^{|\mathcal{P}|}$ and $q(\mathcal{P}') \geq q(\mathcal{P}) + \varepsilon^5$.
\end{lemma}
\begin{proof}
	For each $1 \leq i < j \leq r$ for which $(P_i,P_j)$ is not $\varepsilon$-regular, let $P_{i,j} \subseteq P_i$, $P_{j,i} \subseteq P_j$ be such that 
	$\D(P_{i,j}) \geq \varepsilon \D(P_i), \D(P_{j,i}) \geq \varepsilon \D(P_j)$, and $|d(P_{i,j},P_{j,i}) - d(P_i,P_j)| > \varepsilon$.
	For each $1 \leq i \leq r$, let $\mathcal{P}_i$ be the partition of $P_i$, formed by taking the common refinement of the partitions $\{P_{i,j}, P_i \setminus P_{i,j}\}$, where $j$ runs over all indices for which $(P_i,P_j)$ is not $\varepsilon$-regular. Let $\mathcal{P}' = \bigcup_{i=1}^{r}{\mathcal{P}_i}$ be the resulting refinement of $\mathcal{P}$. Then clearly $|\mathcal{P}'| \leq |\mathcal{P}| \cdot 2^{|\mathcal{P}|}$. We now show that $q(\mathcal{P}') \geq q(\mathcal{P}) + \varepsilon^5$.
	First, observe that by Lemma \ref{lem:density_mean_variance}, for every $1 \leq i < j \leq r$ we have
	$
	\sum_{X' \in \mathcal{P}_i, Y' \in \mathcal{P}_j}{\D(X')\D(Y') \cdot d^2(X',Y')} \geq \mathcal{D}(P_i)\mathcal{D}(P_j) \cdot d^2(P_i,P_j).
	$
	Next, fix any pair $1 \leq i < j \leq r$ for which $(P_i,P_j)$ is not $\varepsilon$-regular. By Lemma \ref{lem:density_mean_variance} we have
	\begin{align*}
	&\sum_{X' \in \mathcal{P}_i, Y' \in \mathcal{P}_j}{\D(X')\D(Y') \cdot d^2(X',Y')}
	= \\&
	\mathcal{D}(P_i)\mathcal{D}(P_j) \cdot d^2(P_i,P_j) +
	\sum_{X' \in \mathcal{P}_i, Y' \in \mathcal{P}_j}
	{\D(X')\D(Y') \cdot
		\left(d(X',Y') - d(P_i,P_j)\right)^2}
	\geq \\&
	\mathcal{D}(P_i)\mathcal{D}(P_j) \cdot d^2(P_i,P_j) +
	\sum_{X' \subseteq P_{i,j}, Y' \subseteq P_{j,i}}
	{\D(X')\D(Y') \cdot \left(d(X',Y') - d(P_i,P_j)\right)^2}
	= \\&
	\mathcal{D}(P_i)\mathcal{D}(P_j) \cdot d^2(P_i,P_j) + 
	\sum_{X' \subseteq P_{i,j}, Y' \subseteq P_{j,i}}
	{\D(X')\D(Y')
		\cdot \left[ \left(d(X',Y') - d(P_{i,j},P_{j,i})\right) +
		\left(d(P_{i,j},P_{j,i}) - d(P_i,P_j)\right) \right]^2}
	\geq \\&
	\mathcal{D}(P_i)\mathcal{D}(P_j) \cdot d^2(P_i,P_j) + 
	\mathcal{D}(P_{i,j})\mathcal{D}(P_{j,i}) \cdot 
	\left(d(P_{i,j},P_{j,i}) - d(P_i,P_j)\right)^2
	\geq
	\mathcal{D}(P_i)\mathcal{D}(P_j) \cdot (d^2(P_i,P_j) + \varepsilon^4),
	\end{align*}
	where in the penultimate inequality we used the first part of Lemma \ref{lem:density_mean_variance} to infer that
	$$
	\sum_{X' \subseteq P_{i,j}, Y' \subseteq P_{j,i}}
	{\D(X')\D(Y') \cdot \left( d(X',Y') - d(P_{i,j},P_{j,i}) \right)} = 0.
	$$
	Denoting by $\mathcal{N}$ the set of pairs $1 \leq i < j \leq r$ for which $(P_i,P_j)$ is not $\varepsilon$-regular, we see that
	\begin{align*}
	q(\mathcal{P}') &\geq
	\sum_{1 \leq i < j \leq r} \;
	{	\sum_{X' \in \mathcal{P}_i, Y' \in \mathcal{P}_j}
		{\D(X')\D(Y') \cdot d^2(X',Y')}}
	\\ &\geq
	\sum_{1 \leq i < j \leq r}
	{\D(P_i)\D(P_j) \cdot d^2(P_i,P_j)} +
	\sum_{(i,j) \in \mathcal{N}}{\D(P_i)\D(P_j) \cdot \varepsilon^4}
	\\ &=
	q(\mathcal{P}) +
	\varepsilon^4 \cdot \sum_{(i,j) \in \mathcal{N}}{\D(P_i)\D(P_j)} \geq
	q(\mathcal{P}) + \varepsilon^5,
	\end{align*}
	where in the last inequality we used the assumption that $\mathcal{P}$ is not $\varepsilon$-regular.
\end{proof}

\begin{proof}[Proof of Lemma \ref{lem:reg}]
	For $i \geq 0$, if $\mathcal{P}_i$ is not $\varepsilon$-regular then we apply Lemma \ref{lem:index_increase} to obtain a partition $\mathcal{P}_{i+1}$ which refines $\mathcal{P}_i$ and satisfies $|\mathcal{P}_{i+1}| \leq |\mathcal{P}_i|\cdot 2^{|\mathcal{P}_i|}$ and $q(\mathcal{P}_{i+1}) \geq q(\mathcal{P}_i) + \varepsilon^5$. Since the index of any partition is at most $1$, this process must end after at most $\varepsilon^{-5}$ steps. When the process ends, we have an $\varepsilon$-regular partition. 
	Since the number of steps depends only on $\varepsilon$,
	the size of the resulting final partition can be upper-bounded by a function of $\varepsilon$ and $|\mathcal{P}_0|$, as required.
\end{proof}
\begin{proof}[Proof of Lemma \ref{lem:strong_reg}]
	We may assume, without loss of generality, that $\mathcal{E}$ is monotone decreasing.
	Let $\mathcal{P}_1$ be the partition obtained by applying Lemma \ref{lem:reg} with parameter $\varepsilon = \mathcal{E}(0)$ and with the partition $\mathcal{P}_0$. Next, for each $i \geq 1$, apply Lemma \ref{lem:reg} with parameter
	$\mathcal{E}(|\mathcal{P}_i|)$
	and with the partition $\mathcal{P}_i$ to obtain a partition $\mathcal{P}_{i+1}$ which is $\mathcal{E}(|\mathcal{P}_i|)$-regular and refines $\mathcal{P}_i$. In light of Lemma \ref{lem:index_refinement}, and as the index of any partition is at most $1$, there must be some $1 \leq i \leq \frac{1}{\mathcal{E}^2(0)}$ for which
	$q(\mathcal{P}_{i+1}) \leq q\mathcal(\mathcal{P}_i) + \mathcal{E}^2(0)$. For such an $i$, set $\mathcal{P} = \mathcal{P}_i$ and $\mathcal{Q} = \mathcal{P}_{i+1}$.
	Since $|\mathcal{P}_0| \leq m$ and the number of steps in the process is at most $\mathcal{E}^2(0)$, and since the size of the partition guaranteed by Lemma \ref{lem:reg} can be bounded from above by a function of the parameters of this lemma (which in our case depend only on $\mathcal{E}$ and $m$), we see that $|\mathcal{Q}|$ too can be bounded from above by a function of $\mathcal{E}$ and $m$. This proves Item 1.
	
	Item 2 is immediate from our choice of $\mathcal{Q}$.
	It remains to prove Item 3. By the definition of the index and by our choice of $\mathcal{P}$ and $\mathcal{Q}$, we have
	\begin{align*}
	&q(\mathcal{P}) + \mathcal{E}^2(0) \geq q(\mathcal{Q}) \geq
	\sum_{P_1,P_2 \in \mathcal{P}}{
		\sum_{Q_1 \subseteq P_1, Q_2 \subseteq P_2}
		{\D(Q_1)\D(Q_2) \cdot d^2(Q_1,Q_2)}
	}
	= \\&
	\sum_{P_1,P_2 \in \mathcal{P}}{\D(P_1)\D(P_2) \cdot d^2(P_1,P_2)} +
	\sum_{P_1,P_2 \in \mathcal{P}}{
		\sum_{Q_1 \subseteq P_1, Q_2 \subseteq P_2}
		{\D(Q_1)\D(Q_2) \cdot \left( d(Q_1,Q_2) - d(P_1,P_2) \right)^2}}
	= \\&
	q(\mathcal{P}) +
	\sum_{P_1,P_2 \in \mathcal{P}}{
		\sum_{Q_1 \subseteq P_1, Q_2 \subseteq P_2}
		{\D(Q_1)\D(Q_2) \cdot \left( d(Q_1,Q_2) - d(P_1,P_2) \right)^2}},
	\end{align*}
	where in the first equality we used the second part of Lemma \ref{lem:density_mean_variance}. The above implies that
	$$
	\sum_{P_1,P_2 \in \mathcal{P}}{
		\sum_{Q_1 \subseteq P_1, Q_2 \subseteq P_2}
		{\D(Q_1)\D(Q_2) \cdot \left( d(Q_1,Q_2) - d(P_1,P_2) \right)^2}} \leq \mathcal{E}^2(0),
	$$
	and hence
	\begin{align*}
	&\sum_{P_1,P_2 \in \mathcal{P}}{
		\sum_{Q_1 \subseteq P_1, Q_2 \subseteq P_2}
		{\D(Q_1)\D(Q_2) \cdot \left| d(Q_1,Q_2) - d(P_1,P_2) \right|}}
	\leq \\&
	\sqrt{\sum_{P_1,P_2 \in \mathcal{P}}{
			\sum_{Q_1 \subseteq P_1, Q_2 \subseteq P_2}
			{\D(Q_1)\D(Q_2) \cdot \left( d(Q_1,Q_2) - d(P_1,P_2) \right)^2}}} \leq \mathcal{E}(0),
	\end{align*}
	where the first inequality follows from Cauchy-Schwarz. This completes the proof.
\end{proof}
\end{document}